\documentclass[11pt]{amsart}
\usepackage{stmaryrd}
\usepackage{amssymb}
\usepackage{csvsimple}
\usepackage{siunitx}
\usepackage{hyperref}
\usepackage{url}
\usepackage{graphicx,color}
\usepackage{lineno}           
\usepackage{enumerate}
\usepackage[style=trad-abbrv,url=false,doi=true,
  eprint=true,isbn=false,backend=biber]{biblatex}

\graphicspath{{pics/},{figs/}}
\def\R{\mathbb{R}}

\def\N{\mathbb{N}}

\def\cF{\mathcal{F}}

\def\cI{\mathcal{I}}

\def\cN{\mathcal{N}}

\def\cS{\mathcal{S}}
\def\cT{\mathcal{T}}
\def\cU{\mathcal{U}}
\def\cV{\mathcal{V}}

\def\cX{\mathcal{X}}

\def\g{\gamma}
\def\d{\delta}

\def\p{\partial}
\def\o{\omega}
\def\veps{\varepsilon}

\def\O{\Omega}
\def\G{\Gamma}
\def\GD{{\Gamma_D}}

\def\transp{{\sf T}}

\def\hu{\widehat{u}}
\def\tu{\widetilde{u}}

\newcommand{\dv}[1]{\,{\mathrm d}#1}

\newcommand{\wcheck}[1]{#1\hspace{-.8ex}\mbox{\huge {\lower.45ex \hbox{$\textstyle \check{}$}}} \hspace{.5ex}}

\DeclareMathOperator{\id}{id}

\let\oldmarginpar\marginpar
\renewcommand\marginpar[1]{
  \oldmarginpar[\raggedleft\footnotesize #1]
  {\raggedright\footnotesize #1}}

\newtheorem{definition}{Definition}
\newtheorem{lemma}[definition]{Lemma}
\newtheorem{proposition}[definition]{Proposition}

\newtheorem{remark}[definition]{Remark}

\newtheorem{example}[definition]{Example}

\newtheorem{algorithm}[definition]{Algorithm}

\numberwithin{definition}{section}
\definecolor{modmag}{RGB}{179,0,229}

\newcommand{\EDT}{\color{black}}
\renewcommand{\text}{\textnormal}

\def\eoc{{\rm eoc}}

\def\bc{{\rm bc}}
\def\hm{{\rm hm}}
\def\uni{{\rm uni}}
\def\ener{{\rm ener}}
\def\stop{{\rm stop}}
\def\iso{{\rm iso}}
\def\bend{{\rm bend}}

\def\l{\mu}
\def\hv{\widehat v}
\usepackage{comment}

\begin{document}
\title[Quadratic constraint consistency]{Quadratic constraint
consistency in the projection-free approximation of harmonic maps
and bending isometries}
\author[G. Akrivis]{Georgios Akrivis}
\address{Department of Computer Science and Engineering,
University of Ioannina, 451$\,$10 Ioannina, Greece, and
Institute of Applied and Computational Mathematics,
FORTH, 700$\,$13 Heraklion, Crete, Greece}
\email{akrivis@cse.uoi.gr}
\author[S. Bartels]{S\"oren Bartels}
\address{Abteilung f\"ur Angewandte Mathematik,
Albert-Ludwigs-Universit\"at Freiburg, Hermann-Herder-Str.~10,
79104 Freiburg i.~Br., Germany}
\email{bartels@mathematik.uni-freiburg.de}
\author[C. Palus]{Christian Palus}
\address{Abteilung f\"ur Angewandte Mathematik,
Albert-Ludwigs-Universit\"at Freiburg, Hermann-Herder-Str.~10,
79104 Freiburg i.~Br., Germany}
\email{christian.palus@mathematik.uni-freiburg.de}
\date{\today}
\renewcommand{\subjclassname}{
\textup{2010} Mathematics Subject Classification}
\begin{abstract}
We devise a projection-free iterative scheme for the approximation
of harmonic maps that provides a second-order accuracy of the
constraint violation and is unconditionally energy stable. A corresponding
error estimate is valid under a mild but necessary discrete regularity
condition. The method
is based on the application of a BDF2 scheme and the considered problem
serves as a model for partial differential equations with
holonomic constraint. The performance of the method is illustrated
via the computation of stationary harmonic maps and bending isometries.
\end{abstract}

\maketitle

\section{Introduction}
A widely used approach to discretizing partial differential equations
that involve a nonlinear pointwise constraint follows~\cite{Alou97} and
is based on semi-implicit discretizations of gradient flows or evolution
problems with a linearized treatment
of the constraint. A corresponding
projection step to guarantee an exact satisfaction of the constraint in
appropriate quadrature points can only be used in special situations, e.g.,
if the problem is of second order and the finite element discretizations
under consideration
provide certain monotonicity properties; cf.~\cite{Bart05}. However, even if the
projection is
stable, it may increase the residual of an approximation. It was
observed in~\cite{Bart16} that the projection step can be omitted
in many situations and that the resulting constraint violation is controlled
(linearly) by the
step size independently of the number of iterations. We refer the reader
to~\cite[Ch.~7]{Bart15-book} for an overview of these results. If a high
accuracy in the approximation of the constraint is desired, then this
limits the efficiency of the numerical method. It is the goal of this
article to devise a variant of the projection-free scheme resulting
from combining~\cite{Alou97} and~\cite{Bart16} that provides second order
accuracy in the constraint violation under {\EDT sharp} discrete regularity conditions
but is guaranteed to satisfy a first order accuracy property
unconditionally. Harmonic maps serve as a model problem for partial differential
equations with holonomic constraint, the application of our results
to other problems is illustrated by the computation of bending isometries.

To explain the main ideas, we consider the numerical approximation
of harmonic maps into spheres that are stationary configurations of
the Dirichlet energy among unit-length vector fields for given
boundary conditions, i.e.,
\[
-\Delta u = \lambda u, \quad
|u|^2 = 1 \ \text{in } \O, \quad
u = u_D \ \text{on } \G_D, \quad
\p_n u = 0 \ \text{on } \p\O \setminus \G_D,
\]
in a bounded Lipschitz domain $\O\subset \R^d$ with boundary part
$\G_D \subset \p\O$ of positive surface measure and
a given function {\EDT $u_D$ which is assumed
to be equal to the trace of a function $\tu_D \in H^1(\O; \R^\ell)$ with
$|\tu_D|=1$ almost everywhere in $\O$}. The function $\lambda$ is the
Lagrange multiplier related to the unit-length constraint and is given
by $\lambda = |\nabla u|^2$; {\EDT see~\cite[Ch.~7]{Bart15-book}
and references therein}. A weak formulation determines a solution
$u\in H^1(\O;\R^\ell)$ with $u = u_D$ on $\G_D$ and $|u|^2 = 1$
in $\O$, via the equation
\begin{equation}\label{eq:hm_weak}
(\nabla u,\nabla v) = 0
\end{equation}
for all $v\in H^1_D(\O;\R^\ell)$ with $v\cdot u = 0$ in $\O$, i.e.,
$u$ is stationary with respect to tangential perturbations on the
unit sphere along $u$; {\EDT see, e.g.,~\cite{Alou97} or
the overview in~\cite[Ch.~7]{Bart15-book} for details on the derivation
of the weak formulation~\eqref{eq:hm_weak}}.
{\EDT The corresponding formulation as a saddle point problem has
been used in various articles as a foundation for the development of
approximation schemes; see,~e.g.,~\cite{QiTaWi09,BadGuiGut11}.}
In view of irregularity results for general
harmonic maps, cf.~\cite{Rivi95}, it is important to compute harmonic
maps with low Dirichlet energy.

The iterative scheme devised in~\cite{Alou97,Bart16} realizes
a semi-implicit time discretization of the gradient flow problem
\[
(\p_t u,v)_\star + (\nabla u,\nabla v) = 0
\]
for all $v\in H^1_D(\O;\R^\ell)$ subject to initial and boundary
conditions $u(0,\cdot)=u^0$ with $|u^0|^2=1$
and $u|_{\G_D} = u_D$, $v|_{\G_D}=0$,  and the constraints
\[
\p_t u \cdot u = 0, \quad v \cdot u = 0.
\]
{\EDT The choice of inner product $(\cdot,\cdot)_\star$ in the gradient
flow problem depends on the particular application, e.g.,~choosing the $L^2$
inner product $(\cdot, \cdot)$ yields approximation of the heat flow of harmonic maps
into the sphere, while choosing the $H^1_D$ inner product $(\nabla \cdot, \nabla \cdot)$
is well-suited for approximating stationary configurations.}
Particularly, with {\EDT the step size $\tau>0$ and} the backward difference quotient operator
$d_t u^n = (u^n-u^{n-1})/\tau$, it computes for given $u^0$
the sequence $(u^n)_{n=1,2,\dots}$ via the sequence of problems
\[
(d_t u^n,v)_\star + (\nabla u^n,\nabla v) = 0
\]
subject to homogeneous boundary conditions for $d_t u^n$ and
$v$ on $\G_D$, and the linearized unit length condition
\[
d_t u^n \cdot u^{n-1} = 0, \quad v \cdot u^{n-1} = 0.
\]
Note that here $d_t u^n$ is seen as the unknown variable which then
defines $u^n$ via $u^n = u^{n-1}+ \tau d_t u^n$.
The iteration is unconditionally well posed and energy decreasing,
i.e., choosing $v = d_t u^n$ yields that
\[
\|d_t u^n\|_\star^2 + \frac12 d_t \|\nabla u^n\|^2 + \frac{\tau}{2} \|\nabla d_t u^n\|^2= 0.
\]
This implies the summability of the discrete time derivatives $\|d_t u^n\|_\star^2$ and
hence the weak convergence of subsequences to solutions of~\eqref{eq:hm_weak}.
A bound for the constraint violation thus follows from the orthogonality
condition and $|u^0|^2 =1$, i.e., we have
\[
|u^n|^2 -1 = |u^{n-1}|^2 -1 + \tau^2 |d_t u^n|^2 = \dots
= \tau^2 \sum_{j=1}^n |d_t u^j|^2.
\]
Taking the $L^1$ norm of this identity, the sum on the right-hand side is bounded by
$\tau (c_\star/2) \|\nabla u^0\|^2$ provided that the induced norm $\|\cdot \|_\star$ controls
the $L^2$ norm up to a factor $c_\star^{1/2}$.

The iterative scheme can also be seen as a backward Euler method for the
$L^2$ flow of harmonic maps if the flow metric is the $L^2$ inner product.
For such evolution problems the discretization based on higher order
time stepping methods has recently been investigated in~\cite{AFKL21};
cf.\ also~\cite{BaKoWa23} for a nodal treatment of the unit-length constraint.
Provided that a sufficiently regular solution exists, quasi-optimal error
estimates have been derived which imply bounds on the constraint violation.
We study here the violation of the constraint in the absence of a smooth and
unique solution. The use of the $H^1$ seminorm $\|\cdot\|_\star = \|\nabla \cdot\|$
defines an $H^1$ gradient flow. {\EDT Being equal to the energy norm,
this scalar product is natural for the minimization problem and acts as a
preconditioner. This typically leads to a faster energy decay and, thus, is
a preferable choice if one is interested in approximating stationary configurations.}

The generalization of the semi-implicit backward Euler method for
the harmonic map heat flow devised in~\cite{AFKL21} computes for
given $u^0,\dots,u^{k-1}$ the sequence $(u^n)_{n=k,k+1,\dots}$ via the scheme
\begin{equation*}
 (\dot{u}^n,v)_\star + (\nabla u^n,\nabla v)= 0
\end{equation*}
subject to homogeneous boundary conditions on $\G_D$ and the linearized constraint
\[
\dot{u}^n \cdot \hu^n = 0, \quad v \cdot \hu^n = 0.
\]
Here $\dot{u}^n$ is a higher order approximation of the time derivative
and $\hu^n$ a suitable explicit extrapolation. Adopting concepts from the construction
of backward differentiation formula (BDF) methods as analyzed in, e.g.,~\cite{HaiWan96,AkrLub15},
approximations with second order consistency properties are given by
\[
\dot{u}^n = \frac{1}{2\tau}\big(3 u^n - 4 u^{n-1} + u^{n-2}\big),
\]
or equivalently $2\dot{u}^n = 3d_t u^n - d_t u^{n-1}$, and
\begin{equation*}
  \hu^n = u^{n-1} + \tau d_t u^{n-1} = 2 u^{n-1} - u^{n-2}.
\end{equation*}
In particular, we have that $u^n = \big(4 u^{n-1} - u^{n-2} + 2 \tau \dot{u}^n\big)/3$.

The iteration is initialized with one step of the linearized  backward
Euler method and then repeated until the discrete time-derivatives are sufficiently
small or some final time $T>0$ is reached. {\EDT This initialization step does not
affect the convergence behavior in the $L^2$ norm; cf.~\cite[Theorem~1.7]{Thom06-book} for related
details.
As an alternative in discrete settings, a first iterate that satisfies a nodal
unit-length constraint exactly may be obtained by nodally projecting the backward
Euler iterate to the unit sphere, or by employing a nonlinear scheme; see, e.g.,~\cite{GutRes17}
for a Crank-Nicolson type method. However, such approaches typically involve
additional constraints, e.g., on underlying triangulations, to guarantee energy stability.}
Note that we always regard $\dot{u}^n$
as the unknown variable in the time steps which is then used to specify the
new iterate $u^n$. The function $\dot{u}^n$ satisfies homogeneous boundary
conditions on $\G_D$ if $u^{n-2},u^{n-1},u^n$ equal $u_D$ on $\G_D$.

\begin{algorithm}\label{alg:bdf2_iter}
Choose $u^0\in H^1(\O;\R^\ell)$ with $u^0|_\GD = u_D$ and $|u^0|^2 = 1$. \\
(0) Compute $d_t u^1 \in H^1_D(\O;\R^\ell)$ such that $d_t u^1 \cdot u^0 = 0$ and
\[
(d_t u^1, v)_\star + (\nabla [u^0+\tau d_t u^1],\nabla v) = 0
\]
for all $v\in H^1_D(\O;\R^\ell)$ with $v\cdot u^0 = 0$; set $u^1 = u^0+\tau d_t u^1$ and $n=2$. \\
(1) Set $\hu^n = 2 u^{n-1} - u^{n-2}$ and
compute $\dot{u}^n \in H^1_D(\O;\R^\ell)$ with $\dot{u}^n \cdot \hu^n = 0$ and
\[
(\dot{u}^n, v)_\star + \frac13 (\nabla [4 u^{n-1} - u^{n-2} + 2 \tau \dot{u}^n],\nabla v) = 0
\]
for all $v\in H^1_D(\O;\R^\ell)$ with $v\cdot \hu^n = 0$; set
$u^n =(4 u^{n-1} - u^{n-2} + 2 \tau \dot{u}^n)/3$. \\
(2) Stop if $\|\dot{u}^n\|_\star + \|d_t u^n\| \le \veps_{\rm stop}$ or $n\tau \ge T$. \\
(3) Increase $n \to n+1$ and continue with~(1).
\end{algorithm}

{\EDT
The stopping criterion in Step~(2) of the algorithm controls the residuals
in the partial differential equation~\eqref{eq:hm_weak} and the involved orthogonality relation; in
particular, it provides control over the difference between $\hu^n$ and $u^n$.
}

Since the subspace of functions $v\in H^1_D(\O;\R^\ell)$ satisfying $v\cdot \hu^n = 0$ in $\O$ is
{\EDT weakly} closed, the Lax--Milgram lemma implies that the iteration is unconditionally well defined
and terminates within a finite number of iterations. More precisely, we
{\EDT show in Proposition~\ref{prop:ener_stab}} that
\[
\|\nabla \cU^N\|_G^2 + \tau \sum_{n=2}^N \|\dot{u}^n\|_\star^2 \le \|\nabla \cU^1\|_G^2,
\]
where $\cU^n = (u^n,u^{n-1})$ and $\|\cdot\|_G$ denotes a BDF-adapted variant of the $L^2$
norm {\EDT defined in Section~\ref{subsec:bdf2norm}}.
{\EDT An elementary calculation, cf. Proposition~\ref{prop:initialize} below, shows that}
we have $\|\nabla \cU^1\|_G \le c_G^{1/2} \|\nabla u^0\|$ so that $\dot{u}^n \to 0$ as
$n \to \infty$. For the constraint violation we {\EDT show in Proposition~\ref{la:constr_vio-mod}
that for $N \ge 2$ we have}
\[{\EDT
\|  |u^N|^2 - 1 \|_{L^1} =\frac 32 \Big (1-\frac 1{3^N}\Big )\tau^2 \|d_tu^1\|^2
+\frac 32\tau^4\sum_{n=2}^N\Big (1-\frac 1{3^{N+1-n}}\Big ) \|d_t^2 u^n\|^2.}
\]

The right-hand side is always of order $O(\tau)$. {\EDT Moreover, if and only if a discrete regularity
property applies, i.e., if and only if $d_t u^1$ belongs to $L^2(\O)$ and the piecewise constant interpolant
of the sequence $\tau^{1/2}d_t^2 u^n$  belongs to $L^2(0,T;L^2(\O))$ uniformly as $\tau \to 0$, then the right-hand
side is of order $O(\tau^2)$.}
All of our results are stated for a semi-discrete
method but hold verbatim if a spatial discretization with a nodal treatment of
the (linearized) constraint is considered.

The article is organized as follows. We specify our notation and collect some
auxiliary results in Section~\ref{sec:prelim}. In Section~\ref{sec:main} we
derive our main result. The application to the computation of
harmonic maps and bending isometries is reported in Section~\ref{sec:exp}.
We remark that other approaches based on higher order time stepping methods for
partial differential equations such as the Landau--Lifshitz--Gilbert equation
typically employ a suitable projection step or make use of constraint-preserving
reformulations; cf.~\cite{BarPro07,AKST14,AnGaSu21,FPPRS20,GuLiWa22,MPPR22}.


\section{Auxiliary results}\label{sec:prelim}
We use standard notation for differential operators and Lebesgue and Sobolev spaces, i.e.,
$H^1_D(\O;\R^\ell)$ denotes the space of vector fields $u:\O\to \R^\ell$ in $L^2(\O;\R^\ell)$
whose weak gradients are square integrable and whose traces vanish on $\G_D\subset \p\O$.
We let $|\cdot|$ denote the Euclidean length of a vector or the Frobenius norm of a matrix
and $\|\cdot \|$ the $L^2$ norm of a function or vector field.

\subsection{Discrete time derivatives}
We always let $\tau>0$ denote a time-step size which gives rise to the backward
difference operator
\[
d_t u^n = \frac{1}{\tau} (u^n - u^{n-1})
\]
for $n=1,2,\dots,N$ and a sequence $(u^n)$ in a Hilbert space. We also make use of
a second discrete time derivative, defined for $n\ge 2$ by
\[
d_t^2 u^n = \frac{1}{\tau^2} (u^n - 2 u^{n-1} + u^{n-2}).
\]
A binomial formula shows that we have
\[
(d_t u^n, u^n) = \frac{d_t}{2} \|u^n\|^2 + \frac{\tau}{2} \|d_t u^n\|^2.
\]
Approximations of time derivatives with higher accuracy can be obtained by a
Lagrange interpolation of $k+1$ successive members of a sequence $(u^n)$ corresponding
to time levels $(t_n)$ and a
subsequent evaluation of the derivative of the interpolation polynomial at $t_n$.
This leads to {\em backward differentiation formulas} and if three successive
values $u^n, u^{n-1}, u^{n-2}$ are used, i.e., $k=2$, provides
the discrete time derivative
\[
\dot{u}^n = \frac{1}{2\tau} \big(3 u^n - 4 u^{n-1} + u^{n-2}\big).
\]
The discrete time derivatives $d_t u^n$ and $\dot{u}^n$ define equivalent
$\ell^2$ seminorms in the sense of the following lemma.

\begin{lemma}[Norm equivalence]\label{la:l2_norm_eq}
For every sequence $(u^n)$ and $N\ge 2$ we have for the seminorms
\[
|(u^n)|_{\tau,1}  = \Big (\tau \sum_{n=2}^N  \|\dot{u}^n\|^2+ \tau \|d_t u^1\|^2\Big )^{1/2},
\quad |(u^n)|_{\tau,2}  = \Big (\tau \sum_{n=1}^N  \|d_t u^n\|^2\Big )^{1/2},
\]
that $c_{12}^{-1} |(u^n)|_{\tau,1} \le |(u^n)|_{\tau,2} \le c_{12} |(u^n)|_{\tau,1}$
with $c_{12}\ge 1$.
\end{lemma}

\begin{proof}
The relation $2 \dot{u}^n= 3d_t u^n-d_t u^{n-1}$ immediately leads to the first
estimate. It also implies the second estimate since
\[
\|d_t u^n\|^2 \le \Big(\frac23 \|\dot{u}^n\| + \frac13 \|d_t u^{n-1}\|\Big)^2
\le \frac89 \|\dot{u}^n\|^2 + \frac29 \|d_t u^{n-1}\|^2.
\]
Summing over $n=2,3,\dots,N$ and absorbing the second sum
on the right-hand side except for $\|d_t u^1\|^2$ implies the estimate.
\end{proof}

We also state an inverse estimate for discrete seminorms.

\begin{lemma}[Inverse estimate]\label{la:inv_est}
For every sequence $(u^n)$ and $N\ge 2$ we have
\[
\Big(\tau \sum_{n=2}^N \|d_t^2 u^n\|^2\Big)^{1/2}
\le \tau^{-1} c_{{\rm inv}} \Big (\tau \sum_{n=2}^N  \|\dot{u}^n\|^2+ \tau \|d_t u^1\|^2\Big )^{1/2}.
\]
\end{lemma}

\begin{proof}
Noting that $2(\dot{u}^n- d_t u^n) =  \tau d_t^2 u^n$ yields that
\[
\tau \|d_t^2 u^n\| \le 2 \big(\|\dot{u}^n\| + \|d_t u^n\|\big).
\]
Taking squares, summing over $n=2,3,\dots,N$,
and incorporating Lemma~\ref{la:l2_norm_eq} proves the estimate.
\end{proof}

\subsection{BDF-adapted norm}\label{subsec:bdf2norm}
The definition of $\dot{u}^n$ leads to the multistep scheme $\dot{y}^n = f(t_n,y^n)$
which has a second order consistency property and is referred to as a {\em BDF2 scheme}.
It satisfies an energy stability property which is a consequence of the identity, cf.~\cite[p.~308]{HaiWan96},
\begin{equation}\label{eq:g_binom}
\dot{u}^n \cdot u^n =  d_t |\cU^n|_G^2 + \frac{\tau^3}{4} |d_t^2 u^n|^2,
\end{equation}
where $\cU^n= (u^n,u^{n-1})$ for $n\ge 1$ and for an arbitrary pair $\cX=(x,y)$
of elements $x,y$ from an inner product space we set
\[
|\cX |_G^2 = (G\cX) \cdot \cX = g_{11} |x|^2 + 2 g_{12} x\cdot y + g_{22}|y|^2,
\]
with $g_{11} = 5/4$, $g_{12}=-1/2$ and $g_{22} = 1/4$. The positive eigenvalues
$\l_\pm = (3\pm 2 \sqrt{2})/4$ of the symmetric matrix $G=(g_{ij})$ yield the equivalence
\[
\l_- (|x|^2 + |y|^2) \le |(x,y)|_G^2 \le \l_+ (|x|^2 + |y|^2).
\]
Moreover, we have $|(x,y)|_G^2 - \frac14\big( |x|^2+ |y|^2\big) =  x \cdot (x-y)$, and
\begin{equation}\label{eq:quad_rel}
|(x,y)|_G^2 - \frac12 |x-y|^2  = \frac34 |x|^2 - \frac14 |y|^2.
\end{equation}
{\EDT The following lemma provides a discrete version of
the identity $(|v|^2)' = 2 v'\cdot v$. Recall that we have
$\hv^n= v^{n-1} + \tau d_t v^{n-1}$.

\begin{lemma}[Discrete chain rule]\label{Le:orthog-time-deriv}
For a sequence $(v^n)$ and $n \ge 2$ we have
\[
2 \dot{v}^n \cdot \hv^n= \frac{1}{2 \tau}
\Big [3|v^n|^2 -4  |v^{n-1}|^2+  |v^{n-2}|^2\Big ]-\frac32 \tau^3 |d_t^2 v^n|^2.
\]
\end{lemma}

\begin{proof}
We start by splitting the left-hand side of the asserted identity as
\[
2 \dot{v}^n \cdot \hv^n= 2 \dot{v}^n \cdot v^n- 2 \dot{v}^n \cdot (v^n-\hv^n).
\]
We apply~\eqref{eq:g_binom} to the first term on the right-hand side. For the
second term we note that $2\dot{v}^n=3d_tv^n-d_tv^{n-1}$ and
$v^n-\hv^n=\tau (d_tv^n-d_tv^{n-1})$, and use the binomial formula
$(3a-b)(a-b) = (a^2-b^2) + 2 (a-b)^2$, i.e.,
\[\begin{split}
2 \dot{v}^n \cdot (v^n-\hv^n)&=  \tau (3d_tv^n-d_tv^{n-1})\cdot (d_tv^n-d_tv^{n-1})\\
&= \tau \big(|d_tv^n|^2-|d_tv^{n-1}|^2\big) + 2\tau^3 |d_t^2 v^n|^2.
\end{split}\]
On combining the identities we deduce with $\cV^n=(v^n,v^{n-1})$ that
\[
2 \dot{v}^n \cdot \hv^n = d_t \big({\EDT 2} |\cV^n|_G^2 - \tau^2 |d_tv^n|^2\big) - \frac32 \tau^3 |d_t^2 v^n|^2.
\]
Incorporating~\eqref{eq:quad_rel} yields the asserted identity.
\end{proof}

\begin{remark}
If $\dot{u}^n\cdot \hu^n = 0$, then we deduce for $n\ge 2$ that
\begin{equation}\label{eq:orthog-time-deriv2}
\frac 32 |u^n|^2 -2  |u^{n-1}|^2+ \frac12  |u^{n-2}|^2=\frac 32 \tau^4 |d_t^2 u^n|^2
\end{equation}
If $|u_\star^n|^2=1$ for all $n\ge 0$, then we have
$2 \dot{u}_\star^n \cdot \hu_\star^n=-\frac 32\tau^3 |d_t^2 u_\star^n|^2$ for $n\ge 2$.
\end{remark}
}

\section{Main result}\label{sec:main}
We provide in this section the derivation of the identities and estimates for the
energy stability and constraint violation. We always denote a pair of subsequent
approximations for $n\ge 1$ via
\[
\cU^n = (u^n,u^{n-1})
\]
with the iterates $(u^n)_{n=0,\dots}$ obtained with Algorithm~\ref{alg:bdf2_iter}.
Throughout the following we assume that the norm induced by the inner product $(\cdot,\cdot)_\star$
controls the $L^2$ norm, i.e., that
\[
\|v\| \le c_\star^{1/2} \|v\|_\star
\]
for all $v\in H^1_D(\O;\R^\ell)$. The first result concerns the initialization step.

\begin{proposition}[Initialization]\label{prop:initialize}
(a) We have
\[
\|\nabla \cU^1\|_G^2 \le c_G \|\nabla u^0\|^2, \quad \tau \|d_t u^1\|_\star^2 \le \frac12 \|\nabla u^0\|^2.
\]
(b) We have
\[
{\EDT \big\| |u^1|^2 - 1 \big\|_{L^1} =  \tau^2 \|d_t u^1\|^2. }
\]
\end{proposition}

\begin{proof}
(a) Choosing $v=d_t u^1$ in Step~(0) of Algorithm~\ref{alg:bdf2_iter} shows that we
have
\[
\frac12 \|\nabla u^1\|^2 + \tau \|d_t u^1\|_\star^2 + \frac{\tau^2}{2} \|\nabla d_t u^1\|^2
= \frac12 \|\nabla u^0\|^2,
\]
which implies the bounds for $\|\nabla \cU^1\|_G^2$ and $\tau \|d_t u^1\|_\star^2$. \\
(b) Since $d_t u^1 \cdot u^0 = 0$ in Step~(0) of Algorithm~\ref{alg:bdf2_iter},
we have that $|u^1|^2 = |u^0|^2 + \tau^2 |d_t u^1|^2$. Noting $|u^0|^2 = 1$ shows the identity.
\end{proof}

The second result implies that the iteration is energy decreasing and that it becomes
stationary for $n\to \infty$.

\begin{proposition}[Energy decay]\label{prop:ener_stab}
For every $N\ge 1$ we have
\[
\|\nabla \cU^N\|_G^2  + \tau  \sum_{n=2}^N \|\dot{u}^n\|_\star^2
+  \frac{\tau^4}{4} \sum_{n=2}^N \|d_t^2 \nabla u^n\|^2
=  \|\nabla \cU^1\|_G^2.
\]
\end{proposition}

\begin{proof}
Choosing $v=\dot{u}^n$ in Step~(1) of Algorithm~\ref{alg:bdf2_iter} yields,
using~\eqref{eq:g_binom}, that
\[
\tau \|\dot{u}^n\|_\star^2 +  \|\nabla \cU^n\|_G^2 - \|\nabla \cU^{n-1}\|_G^2
+ \frac{\tau^4}{4} \|d_t^2 \nabla u^n\|^2 = 0.
\]
A summation over $n=2,3,\dots,N$ leads to the asserted identity.
\end{proof}

\begin{remark}
For the extrapolated value $\hu^{n+1/2} = (3 u^n - u^{n-1})/2$ we have
\[
\frac12 \|\nabla \hu^{n+1/2} \|^2 + \frac18 \tau^2 \|d_t u^n\|^2 = \|\nabla \cU^n\|_G^2,
\]
which yields another version of the energy law and shows that the BDF2 method
has a stabilizing effect.
\end{remark}

We next derive constraint violation estimates which provide an unconditional linear rate
and a quadratic error under a mild {\EDT but necessary}
discrete regularity condition. Qualitatively, the condition
requires that sequences of approximations are uniformly bounded in
$W^{1,\infty}(0,\d;L^2(\O)) \cap H^{3/2}(0,T;L^2(\O))$ for some $\d>0$.

{\EDT
\begin{proposition}[Constraint violation]\label{la:constr_vio-mod}
For every $n\ge 2,$ we have
\[\begin{split}
|u^n|^2 &=-\frac 12 \Big (1-\frac 1{3^{n-1}}\Big )|u^0|^2+\frac 32 \Big (1-\frac 1{3^n}\Big )|u^1|^2
+\frac 32\tau^4\sum_{i=2}^n\Big (1-\frac 1{3^{n+1-i}}\Big ) |d_t^2 u^i|^2.
\end{split}\]
If $|u^0|^2 = 1$, and $u^1$ is computed by the linearly implicit Euler method,
then $\big (|u^n|\big )$ is increasing almost everywhere in~$\O$ and we have
\begin{equation} \label{eq:constr-viol-new6}
\big \|  |u^n|^2 - 1 \big \|_{L^1} =\frac 32 \Big (1-\frac 1{3^n}\Big )\tau^2 \|d_tu^1\|^2
+\frac 32\tau^4\sum_{i=2}^n\Big (1-\frac 1{3^{n+1-i}}\Big ) \|d_t^2 u^i\|^2.
\end{equation}
(a) Unconditionally and uniformly in $n\ge 1$,~\eqref{eq:constr-viol-new6} is bounded by $c_1 \tau$. \\
(b) {\EDT If and only if} for $m\ge 1$ we have
\begin{equation}\label{eq:mild_reg}
\|d_t u^1\|^2  + \tau^2 \sum_{i=2}^{m}  \|d_t^2u^i\|^2 \le c_r,
\end{equation}
then~\eqref{eq:constr-viol-new6} is bounded by $c_2 \tau^2$ for every $n=1,2,\dots,m$,
as $\tau \to 0$.
\end{proposition}

\def\ha{\widetilde{\alpha}}

\begin{proof}
The main idea is to interpret~\eqref{eq:orthog-time-deriv2} as
an inhomogeneous linear difference equation that implies the asserted identity.
The roots of the polynomial $\ha (z)=\frac 32 -2z+\frac 12 z^2$
are $z_1=1$ and $z_2=3,$ and thus the rational function $1/\ha$
is holomorphic in the open unit disk in the complex plane. A Taylor expansion about
the origin yields for $|z|<1$ that
\begin{equation}
\label{eq:expansion}
\frac 1{\ha(z)}
= \sum_{n=0}^\infty \gamma_n z^n.  
\end{equation}
Multiplying this identity by $\ha(z)$ and comparing coefficients leads to the values
$\g_0 = 2/3$ and $\g_1 = 8/9$, and the recursion formula
\[
\frac 32\gamma_n-2\gamma_{n-1}+\frac 12\gamma_{n-2}=0.
\]
Noting that $(\g_n)$ is given as a linear combination of the sequences
$z_1^{-n}=1$ and $z_2^{-n} = 3^{-n}$ shows that $\g_n = 1- 3^{-(n+1)}$, $n\ge 0$.
We next consider the difference equation \eqref{eq:orthog-time-deriv2} with $n$
replaced by $n-j, j=0,\dots,n-2,$ multiply the corresponding equations by
$\gamma_j,$ and sum over $j$ to derive the identity
\[ 
\sum_{j=0}^{n-2}\gamma_j\Big (  \frac 32 |u^{n-j}|^2 -2  |u^{n-j-1}|^2+ \frac12  |u^{n-j-2}|^2\Big )=
\frac 32\tau^4\sum_{j=0}^{n-2}\Big (1-\frac 1{3^{j+1}}\Big ) |d_t^2 u^{n-j}|^2.
\]
We re-arrange the left-hand side as
\[\begin{split}
& \sum_{j=2}^{n-2} \Big( \frac32 \g_j - 2 \g_{j-1}  + \frac12 \g_{j-2}\Big)  |u^{n-j}|^2 \\
&\quad + \frac32 \g_0 |u^n|^2 + \Big(\frac32 \g_1 - 2 \g_0\Big)|u^{n-1}|^2
- \Big(2\g_{n-2}-\frac12 \g_{n-3}\Big) |u^1|^2 + \frac12 \g_{n-2}|u^0|^2,
\end{split}\]
and use the identities for the coefficients to deduce the asserted identity for $|u^n|^2$,
which immediately leads to~\eqref{eq:constr-viol-new6} noting that $|u^1|^2 = |u^0|^2+ \tau^2 |d_tu^1|^2$. \\
(a) Proposition~\ref{prop:initialize} shows that $\tau \|d_t u^1\|^2$ is uniformly bounded.
The inverse estimate of Lemma~\ref{la:inv_est} in combination with the energy stability
established in Proposition~\ref{prop:ener_stab} thus proves the unconditional estimate. \\
(b) The assumed bound directly leads to the quadratic error estimate. It is optimal since
the coefficients in~\eqref{eq:constr-viol-new6} are uniformly bounded from below.
\end{proof}
}

\begin{remark}\label{rem:time_der}
(i) Choosing the test functions $v=d_tu^1$ and $v=\dot{u}^n$ in Steps~(0) and~(1)
of Algorithm~\ref{alg:bdf2_iter}, respectively, yields that
\[
\|d_t u^1\|_\star^2 \le \|\nabla u^1\| \| \nabla d_t u^1\|, \quad
\|\dot{u}^n\|_\star^2 \le \|\nabla u^n\| \| \nabla \dot{u}^n\|.
\]
Hence, if the norm $\|\cdot\|_\star$ controls the $H^1$ norm, we have that
$\|d_t u^1\|_{H^1}$ and $\|\dot{u}^n\|_{H^1}$, $n\ge 2$, are bounded by the initial energy.
Noting $2 \dot{u}^n = 3d_t u^n- d_tu^{n-1}$ then implies a uniform bound on {\EDT $\|d_t u^n\|_{H^1}$},
$n\ge 1$. \\
(ii) If the norm $\|\cdot\|_\star$ controls the $L^\infty$ norm, e.g., via suitable
Sobolev inequalities or inverse estimates in a spatially discrete setting, then a pointwise
bound for the constraint violation error can be deduced.
\end{remark}

\section{Experiments}\label{sec:exp}
We report in this section the performance of the devised method
used as an iterative procedure to determine stationary configurations
for the pointwise constrained Dirichlet energy and a nonlinear
bending functional. The algorithm devised and analyzed for approximating
harmonic maps into spheres can be greatly generalized and applies to the
numerical solution of a constrained minimization problem
\[
\text{Minimize} \quad I[u] = \frac12 a(u,u) - b(u), \quad u\in V,
\]
subject to boundary conditions $\ell_\bc (u) = u_D$ and a constraint
\[
G(u) = 0.
\]
Given some approximation $\hu \in V$ satisfying $\ell_\bc(\hu) = u_D$ we define
a corresponding linear space via
\[
\cF[\hu] = \{v\in V: \ell_\bc(v) = 0, \, g(\hu;v) = 0 \big\},
\]
where $g$ is the derivative of $G$. Our algorithm then reads as follows.

\begin{algorithm}\label{alg:bdf2_iter_gen}
Choose $u^0\in V$ with $\ell_\bc(u^0) = u_D$ and $G(u^0)=0$. \\
(0) Compute $d_t u^1 \in \cF[u^0]$ with
\[
(d_t u^1, v)_\star + a(u^0+\tau d_t u^1,v) = b(v)
\]
for all $v\in \cF[u^0]$; set $u^1 = u^0+\tau d_t u^1$ and $n=2$. \\
(1) Set $\hu^n = 2 u^{n-1} - u^{n-2}$ and
compute $\dot{u}^n \in \cF[\hu^n] = 0$ with
\[
(\dot{u}^n, v)_\star + \frac13 a(4 u^{n-1} - u^{n-2} + 2 \tau \dot{u}^n,v) = b(v)
\]
for all $v\in \cF[\hu^n]$; set
$u^n =(4 u^{n-1} - u^{n-2} + 2 \tau \dot{u}^n)/3$. \\
(2) Stop if $\|\dot{u}^n\|_\star + \|d_t u^n\|_\sharp \le \veps_{\rm stop}$ or $n\tau \ge T$. \\
(3) Increase $n \to n+1$ and continue with~(1).
\end{algorithm}

{\EDT Choosing $b=0$, $a(\cdot,\cdot) = (\nabla \cdot, \nabla \cdot)$ and $G(u) = |u|^2 -1$ yields the setting of harmonic maps, whereas choosing $a(\cdot,\cdot) = (D^2 \cdot, D^2 \cdot)$
and $G(u)$ to realize an isometry constraint yields the setting of bending isometries
considered below in Section~\ref{subsec:bending_exp}.}
We refer the reader to~{\EDT \cite[Section~4.3.2]{Bart15-book}} for a discussion of admissible
functions $G$ that lead to a constraint violation as discussed above. The
norm $\|\cdot\|_\sharp$ {\EDT is assumed to be} sufficiently strong to provide
control over the linearization error in the constraint.

{\EDT Assuming that $V = V_h$ is a finite element space, we let $\mathbf{\dot{u}}^n \in \R^N$ denote the coefficient vector representing $\dot{u}_h^n \in V_h$
in a suitable basis. In this case Step~(1) is equivalent to the saddle point problem
\begin{equation*}
  \begin{bmatrix}
    \mathbf{A} & [\mathbf{G}^n]^\transp \\
    \mathbf{G}^n & 0
  \end{bmatrix}
  \begin{bmatrix}
    \mathbf{\dot{u}}^n \\
    \boldsymbol{\lambda}
  \end{bmatrix}
  =
  \begin{bmatrix}
    \mathbf{b} \\
    0
  \end{bmatrix},
\end{equation*}
where $\mathbf{A} \in \R^{N \times N}$ encodes the bilinear form $a(\cdot,\cdot)$,
$\mathbf{b}$ is a representation of the right-hand side and explicit terms,
and $\mathbf{G}^n\in \R^{M \times N}$, $M \in \N$, realizes the linear
constraint $g(\widehat{u}_h^n;\cdot)=0$ with $\boldsymbol{\lambda} \in \R^N$ being the
corresponding Lagrange  multiplier. The dimension $M$ of the discrete constraint
 map depends on the constraint discretization, e.g., with the restriction of a one-dimensional constraint to the vertices of a triangulation, $M$ equals the number of vertices in the
triangulation. In our implementation we employ a direct solver for the solution of the
linear system in every time step to obtain $\mathbf{\dot{u}}^n$, which in turn yields a new iterate.
We refer to~\cite{GutRes17,KPPRS19} for results concerning the stability of the constrained
formulation and implementations based on explicit constructions of bases
of the kernel of $\mathbf{G}^n$. }

\subsection{Harmonic maps}\label{subsec:hm_exp}
We define harmonic maps as stationary configurations for the Dirichlet
energy
\[
I_\hm(u) = \frac12 \int_\O |\nabla u|^2 \dv{x}
\]
in the set of mappings $u\in H^1(\O;\R^\ell)$ for $\O\subset \R^d$
satisfying the pointwise unit-length constraint
\[
|u|^2-1 = 0
\]
almost everywhere in $\O$ and the boundary condition $u|_{\G_D} = u_D$
on a subset $\G_D\subset \p\O$ with positive surface measure. For
an extension $u^0 \in H^1(\O;\R^\ell)$ of $u_D$, Algorithm~\ref{alg:bdf2_iter}
determines a sequence $(u^n)$ that converges to a harmonic map of
low energy. In a discrete setting we use the conforming
finite element spaces
\[
V_h = \cS^1(\cT_h)^\ell,
\]
consisting of elementwise affine, globally continuous functions,
and impose the initial unit-length and subsequent orthogonality relations in the
nodes $z\in \cN_h$ of the triangulation $\cT_h$. To compute certain
error quantities we employ the corresponding nodal interpolation
operator $\cI_h:C(\overline{\O};\R^\ell) \to \cS^1(\cT_h)^\ell$.
The results established
for Algorithm~\ref{alg:bdf2_iter} carry over nearly verbatim to
its discrete counterpart; cf.~\cite{Bart05,Bart16}.
We test its performance for a setting leading to a smooth harmonic map.
Experiments for harmonic maps with singularities led to similar results.

\begin{example}[Stereographic projection]\label{ex:invstereo}
For $d=2$, $\ell=3$ we set $\O = (-1/2,1/2)^2$, $\G_D = \p\O$, and $u_D = \pi_{\rm st}^{-1}|_{\p\O}$
with the inverse stereographic projection $\pi_{\rm st}^{-1}: \O \to S^2$ given for $x\in \O$ by
\[
\pi_{\rm st}^{-1}(x) = (|x|^2 + 1)^{-1} \begin{bmatrix} 2 x \\ 1 - |x|^2 \end{bmatrix}.
\]
Then $u = \pi_{\rm st}^{-1}$ is a smooth harmonic map satisfying $u|_{\p\O} = u_D$.
\end{example}

\begin{figure}
\centering
\includegraphics[width=0.32\linewidth]{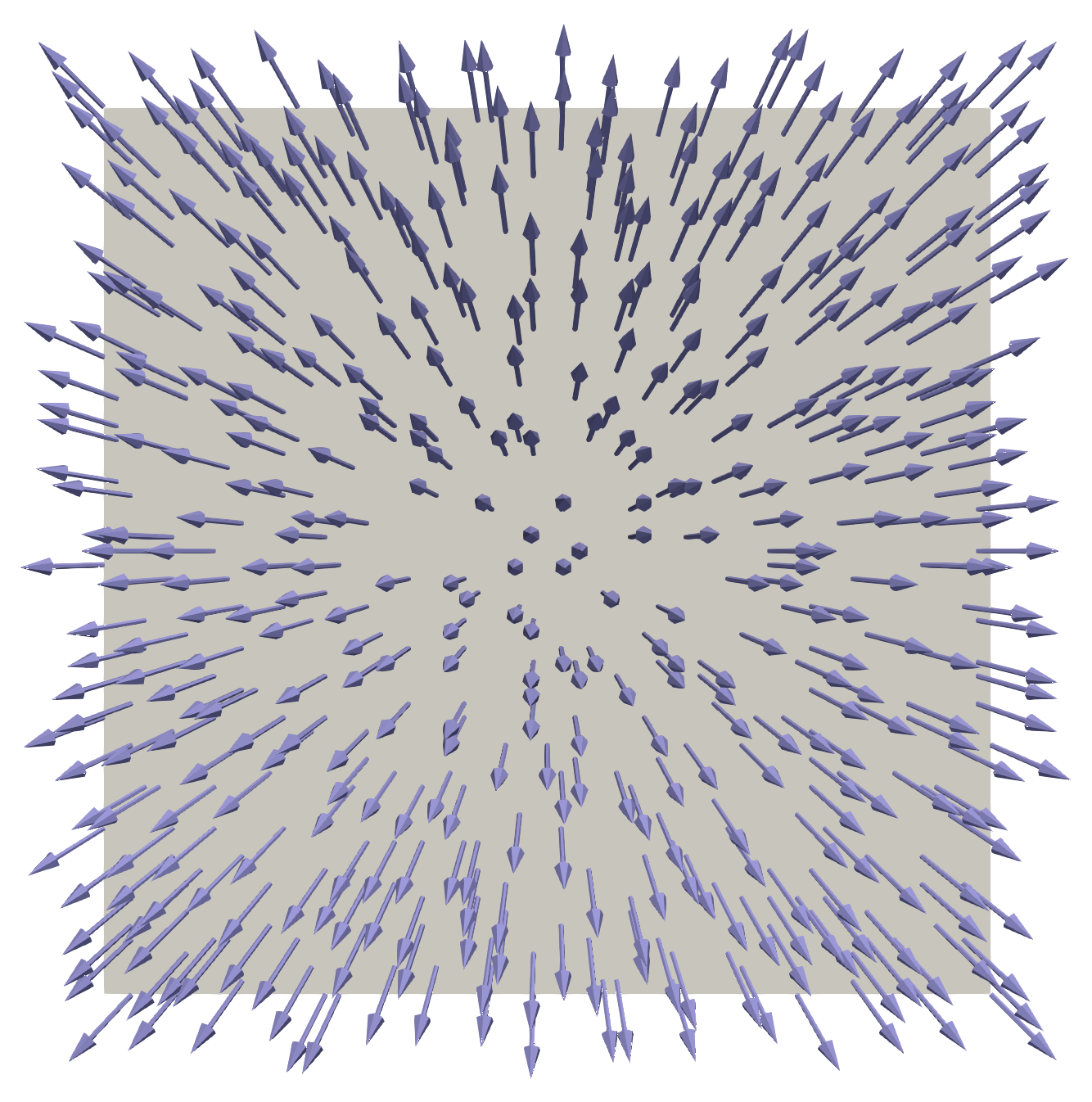}
\includegraphics[width=0.32\linewidth]{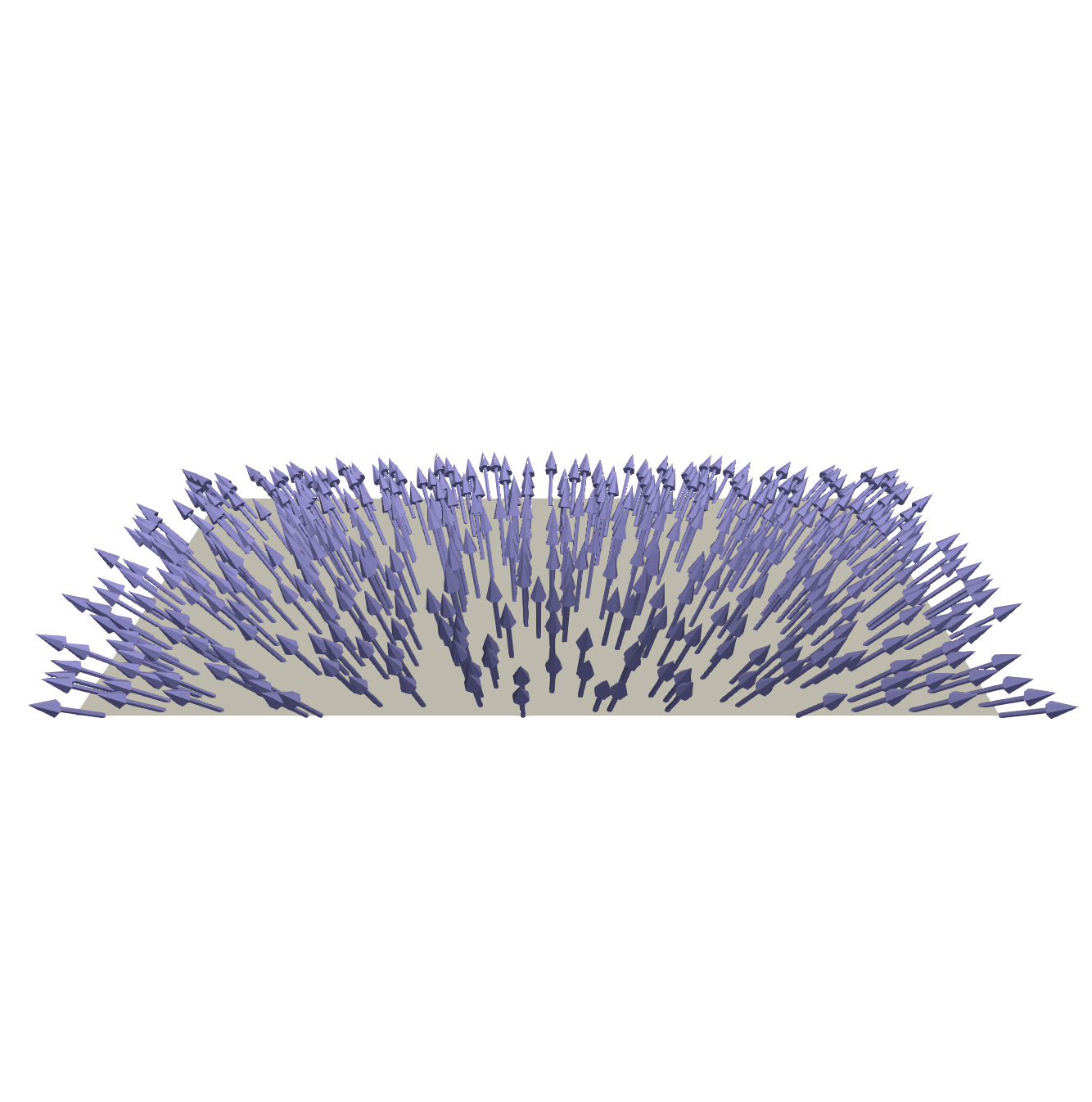}
\includegraphics[width=0.32\linewidth]{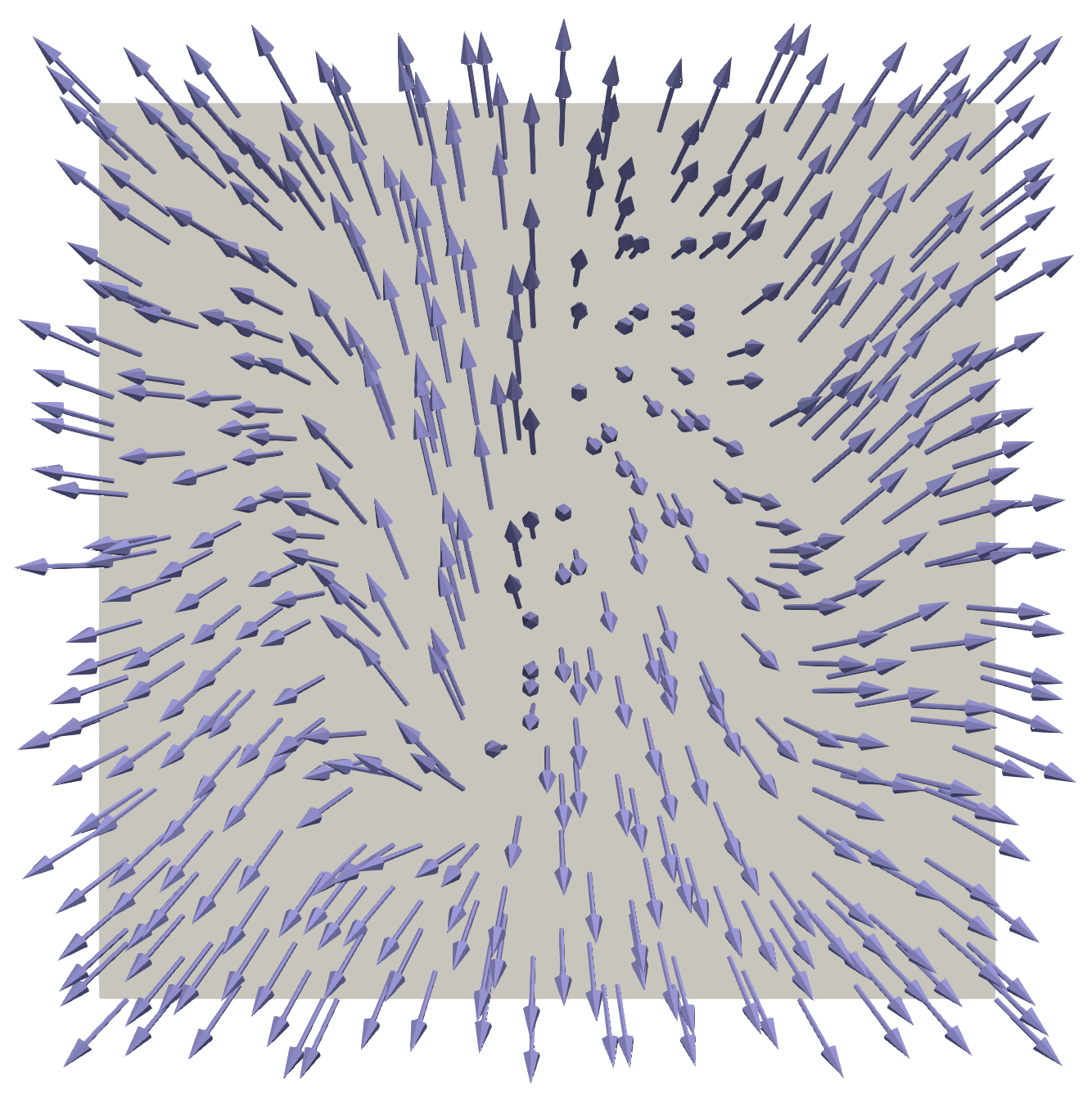}
\caption{\label{fig:invstereo} Nodal interpolant of the inverse stereographic
projection $\pi_{\rm st}^{-1}$ in Example~\ref{ex:invstereo} on a coarse grid (left and middle)
and {\EDT perturbed} initial configuration $u_h^0$ (right).}
\end{figure}

\subsubsection*{Bounded initial data}
The solution $u$ is illustrated in the left and middle plots of Figure~\ref{fig:invstereo}.
For a spatial discretization we choose a uniform triangulation $\cT_h$ of $\O$ into $8192$ right-angled triangles.
The initial function $u_h^0$ and the discrete boundary data $u_{D,h}$ are obtained via a nodal interpolation
of the exact solution $u$ and a subsequent perturbation of interior nodal values; cf.\ the right
plot of Figure~\ref{fig:invstereo}. For this discrete perturbation we have $I_\hm(u_h^0) \approx 22.06$
whereas the exact optimal energy is given by $I_\hm(u) \approx 3.009$.

Using step sizes $\tau = 2^{-m}$, the fixed stopping criterion
$\veps_\mathrm{stop} = 10^{-3}$ in combination with the $L^2$ norm that specifies $\|\cdot\|_\sharp$,
and choosing the $L^2$ and $H^1_D$ inner products for the
gradient flow metric $(\cdot,\cdot)_\star$, respectively, we obtained the results shown in Table~\ref{tab:invstereoh1l2}.
The function $u_h^\stop$ denotes the iterate after $N_\stop$ steps for which the stopping
criterion was satisfied first. The tables show the number of iterations $N_\stop$,
the constraint violation measure
\[
\delta_\uni[u_h] = \big\| \cI_h\left(|u_h|^2 - 1\right) \big\|_{L^1},
\]
the energy errors
\[
\delta_\ener[u_h] = \big| I_\hm[u_h] - I_\hm[u]\big|,
\]
and the discrete regularity quantities
\[
A^2 = \tau^2 \sum_{n=2}^{N'} \|d_t^2 u_h^n\|^2, \quad B^2 = \|d_t u_h^1\|^2,
\]
with $N'=N_\stop$, whose boundedness is needed to guarantee the quadratic constraint consistency results.
The experimental convergence rates $\eoc_\uni$ and $\eoc_\ener$ were computed as logarithmic slopes.

Our observations are as follows: (i) the numbers of iterations to meet the stopping criterion increase
linearly with the decreasing step size and are comparable for the implicit Euler and BDF2 method as
well as for the $L^2$ and $H^1$ gradient flows; (ii) the discrete regularity condition appears to be
satisfied for the $L^2$ and $H^1$ gradient flows and both numerical methods, although a small growth
for the quantity $B^2$ is observed in the case of the $L^2$ flows; (iii) the constraint violation and
energy errors decay linearly for the implicit Euler and (nearly) quadratically for the BDF2 methods
before spatial discretization errors dominate the energy error. Our explanation for~(i) is that the
gradient flows and stopping criteria determine times $t_\star$ at which the time derivative is
sufficiently small and approximately $n\approx t_\star/\tau$ iterations are needed to reach this point
via the time stepping method realized by the algorithm. While the $H^1$ flow provides strong control
over the time derivative, cf.~Remark~\ref{rem:time_der}, the rough initial data appear to lead to
a certain initial growth of the first time derivative in case of the $L^2$ flows providing an
explanation for~(ii). The growth of the time derivative leads to a slight initial reduction
of the quadratic convergence rate reported in~(iii).

  \begin{table}[p]
    {\footnotesize
    \begin{tabular}{ c r c c c c c c c }
      $\tau$ & $N_\stop$ & $\delta_\uni[u_h^\stop]$ & $\mathrm{eoc}_\uni$
        & $A^2$ & $B^2$ & $I_\hm[u_h^\stop]$ & $\delta_\ener[u_h^\stop]$ & $\mathrm{eoc}_\ener$\\ \hline\hline
      \multicolumn{9}{c}{ Implicit Euler method ($L^2$-gradient flow) }\\
      $2^{-4}$  & 23    & 1.288e--01 &  --- & 1.949e+01 & 2.738e+01 & 4.667 & 1.658e+00  &    ---    \\
      $2^{-5}$  & 31    & 1.130e--01 & 0.19 & 5.992e+01 & 9.432e+01 & 4.442 & 1.433e+00  & 0.21\\
      $2^{-6}$  & 47    & 9.219e--02 & 0.29 & 1.594e+02 & 2.937e+02 & 4.133 & 1.124e+00  & 0.35\\
      $2^{-7}$  & 78    & 6.867e--02 & 0.42 & 3.488e+02 & 7.921e+02 & 3.777 & 7.681e--01 & 0.55\\
      $2^{-8}$  & 140   & 4.619e--02 & 0.57 & 6.077e+02 & 1.791e+03 & 3.453 & 4.443e--01 & 0.79\\
      $2^{-9}$  & 261   & 2.823e--02 & 0.71 & 8.446e+02 & 3.375e+03 & 3.229 & 2.201e--01 & 1.01\\
      $2^{-10}$ & 502   & 1.601e--02 & 0.82 & 9.606e+02 & 5.423e+03 & 3.108 & 9.877e--02 & 1.15\\
      $2^{-11}$ & 982   & 8.627e--03 & 0.89 & 9.222e+02 & 7.648e+03 & 3.052 & 4.308e--02 & 1.19\\
      $2^{-12}$ & 1941  & 4.502e--03 & 0.94 & 7.700e+02 & 9.734e+03 & 3.028 & 1.910e--02 & 1.17\\
      $2^{-13}$ & 3858  & 2.305e--03 & 0.97 & 5.711e+02 & 1.145e+04 & 3.018 & 8.678e--03 & 1.13\\
      $2^{-14}$ & 7693  & 1.167e--03 & 0.98 & 3.812e+02 & 1.272e+04 & 3.013 & 3.970e--03 & 1.12\\ \hline
%
      \multicolumn{9}{c}{ BDF2 method ($L^2$-gradient flow) }\\
      $2^{-4}$  & 39    & 2.878e--01 &  --- & 2.125e+01 & 2.738e+01 & 6.518 & 3.509e+00  &    ---    \\
      $2^{-5}$  & 79    & 2.204e--01 & 0.39 & 5.448e+01 & 9.432e+01 & 6.493 & 3.484e+00  &  0.01\\
      $2^{-6}$  & 39    & 1.615e--01 & 0.45 & 1.420e+02 & 2.937e+02 & 5.649 & 2.640e+00  &  0.40\\
      $2^{-7}$  & 76    & 1.035e--01 & 0.64 & 3.229e+02 & 7.921e+02 & 4.557 & 1.548e+00  &  0.77\\
      $2^{-8}$  & 144   & 5.535e--02 & 0.90 & 5.905e+02 & 1.791e+03 & 3.703 & 6.935e--01 &  1.15\\
      $2^{-9}$  & 275   & 2.464e--02 & 1.16 & 8.508e+02 & 3.375e+03 & 3.245 & 2.356e--01 &  1.55\\
      $2^{-10}$ & 534   & 9.370e--03 & 1.39 & 9.793e+02 & 5.423e+03 & 3.074 & 6.448e--02 &  1.86\\
      $2^{-11}$ & 1053  & 3.156e--03 & 1.57 & 9.324e+02 & 7.648e+03 & 3.025 & 1.586e--02 &  2.02\\
      $2^{-12}$ & 2096  & 9.710e--04 & 1.70 & 7.671e+02 & 9.734e+03 & 3.013 & 3.718e--03 &  2.09\\
      $2^{-13}$ & 4184  & 2.794e--04 & 1.79 & 5.635e+02 & 1.145e+04 & 3.010 & 6.977e--04 &  2.41\\
      $2^{-14}$ & 8363  & 7.647e--05 & 1.86 & 3.754e+02 & 1.272e+04 & 3.009 & 7.961e--05 &  3.13\\  \hline\hline
%
      \multicolumn{9}{c}{ Implicit Euler method ($H^1$-gradient flow) }\\
      $2^{-0}$ & 22     & 5.631e--02 &  --- & 1.190e--02 & 3.199e--02 & 3.468 & 4.592e--01 &    ---    \\
      $2^{-1}$ & 32     & 3.429e--02 & 0.72 & 1.351e--02 & 5.686e--02 & 3.236 & 2.265e--01 &  1.02 \\
      $2^{-2}$ & 53     & 1.921e--02 & 0.84 & 1.127e--02 & 8.188e--02 & 3.112 & 1.029e--01 &  1.14 \\
      $2^{-3}$ & 95     & 1.021e--02 & 0.91 & 7.554e--03 & 1.011e--01 & 3.055 & 4.601e--02 &  1.16 \\
      $2^{-4}$ & 178    & 5.270e--03 & 0.95 & 4.424e--03 & 1.133e--01 & 3.030 & 2.097e--02 &  1.13 \\
      $2^{-5}$ & 344    & 2.678e--03 & 0.98 & 2.402e--03 & 1.203e--01 & 3.019 & 9.751e--03 &  1.11 \\
      $2^{-6}$ & 677    & 1.350e--03 & 0.99 & 1.252e--03 & 1.240e--01 & 3.014 & 4.547e--03 &  1.10 \\
      $2^{-7}$ & 1342   & 6.778e--04 & 0.99 & 6.397e--04 & 1.260e--01 & 3.011 & 2.057e--03 &  1.14 \\
      $2^{-8}$ & 2672   & 3.396e--04 & 1.00 & 3.233e--04 & 1.269e--01 & 3.010 & 8.412e--04 &  1.29 \\
      $2^{-9}$ & 5333   & 1.700e--04 & 1.00 & 1.625e--04 & 1.274e--01 & 3.009 & 2.411e--04 &  1.80 \\
      $2^{-10}$ & 10655 & 8.503e--05 & 1.00 & 8.147e--05 & 1.277e--01 & 3.009 & 5.704e--05 &  2.08 \\ \hline
%
      \multicolumn{9}{c}{ BDF2 method ($H^1$-gradient flow) }\\
        $2^{-0}$ & 16     & 7.048e--02  &  --- & 1.458e--02 & 3.199e--02 & 3.570 & 5.604e--01 &    ---    \\
        $2^{-1}$ & 22     & 2.809e--02  & 1.33 & 1.735e--02 & 5.686e--02 & 3.162 & 1.533e--01 &  1.87 \\
        $2^{-2}$ & 44     & 9.071e--03  & 1.63 & 1.396e--02 & 8.188e--02 & 3.046 & 3.712e--02 &  2.05 \\
        $2^{-3}$ & 86     & 2.599e--03  & 1.80 & 8.722e--03 & 1.011e--01 & 3.018 & 9.086e--03 &  2.03 \\
        $2^{-4}$ & 171    & 6.991e--04  & 1.89 & 4.814e--03 & 1.133e--01 & 3.011 & 2.079e--03 &  2.13 \\
        $2^{-5}$ & 341    & 1.817e--04  & 1.94 & 2.515e--03 & 1.203e--01 & 3.009 & 2.706e--04 &  2.94 \\
        $2^{-6}$ & 682    & 4.635e--05  & 1.97 & 1.283e--03 & 1.240e--01 & 3.009 & 1.953e--04 &  0.47 \\
        $2^{-7}$ & 1363   & 1.171e--05  & 1.99 & 6.476e--04 & 1.260e--01 & 3.009 & 3.140e--04 & -0.69 \\
        $2^{-8}$ & 2725   & 2.942e--06  & 1.99 & 3.253e--04 & 1.269e--01 & 3.009 & 3.440e--04 & -0.13 \\
        $2^{-9}$ & 5450   & 7.374e--07  & 2.00 & 1.630e--04 & 1.274e--01 & 3.009 & 3.516e--04 & -0.03 \\
        $2^{-10}$ & 10899 & 1.846e--07  & 2.00 & 8.160e--05 & 1.277e--01 & 3.009 & 3.535e--04 & -0.01 \\  \hline\hline
    \end{tabular}
     }
    \caption{Step sizes, number of iterations, constraint violation, discrete regularity measures, and energy
      errors for the implicit Euler and BDF2 methods approximating $L^2$ and $H^1$ gradient flows for
      harmonic maps {\EDT initialized with a perturbation of the exact solution} in Example~\ref{ex:invstereo}.}
    \label{tab:invstereoh1l2}
  \end{table}

{\EDT
\subsubsection*{Rough initial data}
As a second test problem, we reconsider Example~\ref{ex:invstereo}, again choosing a uniform triangulation $\cT_h$ of $\O$ into $8192$ right-angled triangles.
We now use randomly generated nodal values in the interior of $\O$ instead of a perturbation of the exact solution to initialize the discrete gradient flows, cf.~Figure~\ref{fig:random}.

\begin{figure}
\centering
\includegraphics[width=0.3\linewidth]{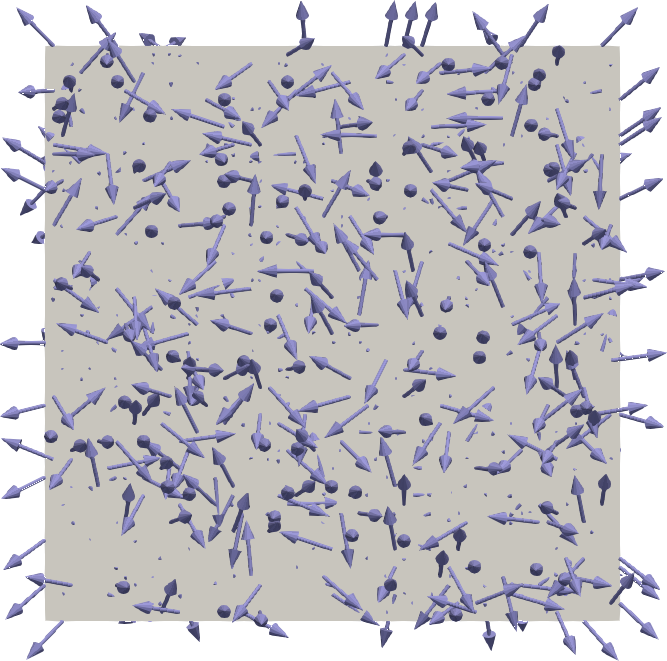}\hspace{0.01\linewidth}
\includegraphics[width=0.3\linewidth]{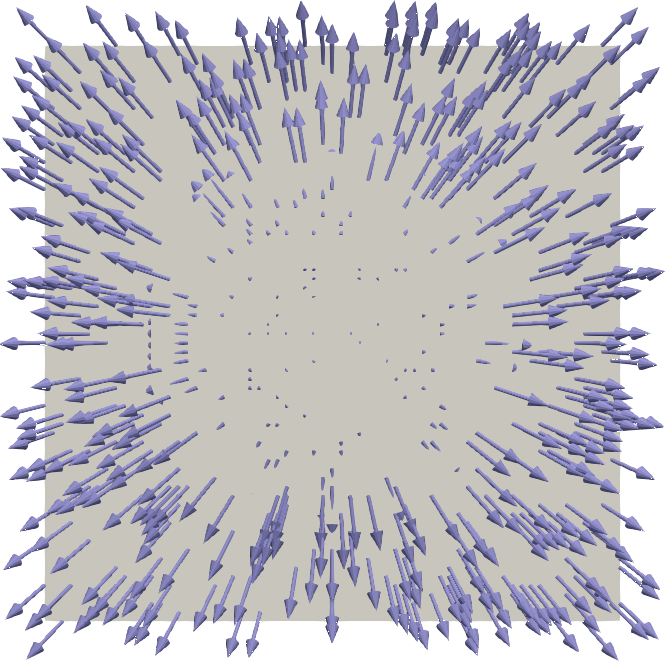}\hspace{0.01\linewidth}
\includegraphics[width=0.3\linewidth]{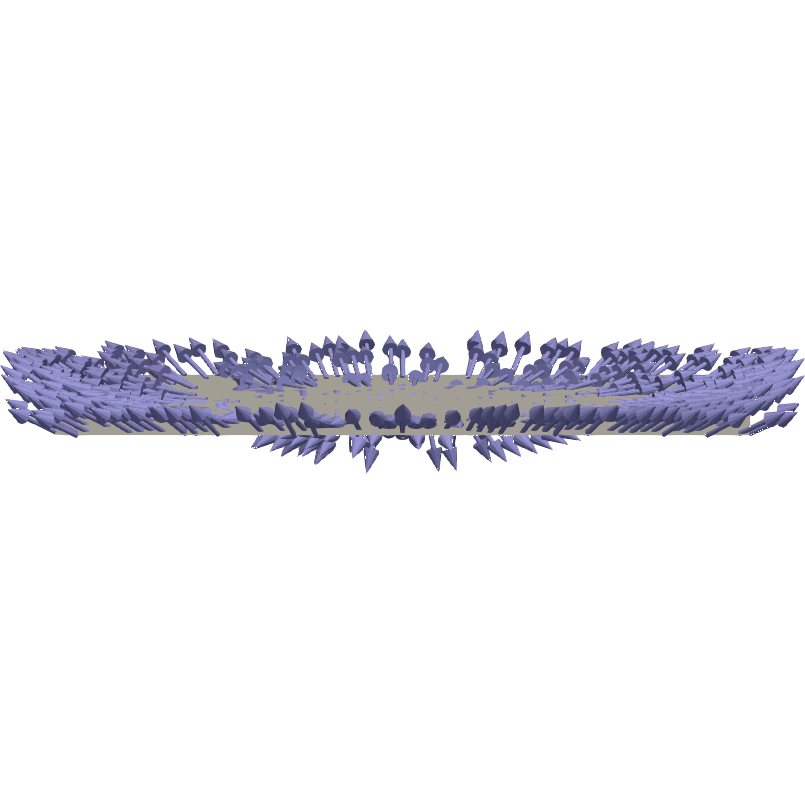}
\caption{\label{fig:random} Nodal values of the rough initial data $u_h^0$ for the approximation of Example~\ref{ex:invstereo} on a coarse grid (left) and final iterate for the BDF2 method with $\tau=2^{-8}$ (middle, right).}
\end{figure}

The initial function $u_h^0$ is obtained by choosing the inverse stereographic projection $u_h^0(z) = \pi_{\rm st}^{-1}(z)$ for all boundary nodes $z \in \cN_h \cap \G_D$ and choosing a random value for $u_h^0(z)$ at all interior nodes $z \in \cN_h \cap \O$.
This random choice is realized via generating for each interior node $z$ two independent and uniformly distributed pseudo-random numbers $\alpha_1 \in (-\pi/2,\pi/2)$, $\alpha_2 \in (-\pi,\pi)$
that we interpret as the angles in the representation of $u_h^0(z)$ in spherical coordinates with the fixed radial distance $1$.
This procedure is conducted only once, such that we use the \emph{same} randomly chosen $u_h^0$ for every step size.
An interpolation of the discrete function $u_h^0$ on a coarse grid is illustrated in the left plot in Figure~\ref{fig:random}.

Using step sizes $\tau = 2^{-m}$, the fixed stopping criterion
$\veps_\mathrm{stop} = 10^{-3}$ in combination with the $L^2$ norm that specifies $\|\cdot\|_\sharp$,
and choosing the $L^2$ and $H^1_D$ inner products for the
gradient flow metric $(\cdot,\cdot)_\star$, respectively, we obtained the results shown in Table~\ref{tab:randomh1l2}.

Partially in contrast to the first experiment with bounded initial data, we observe the following:
(i)~while for the $H^1$ flow the iteration numbers to meet the stopping criterion increase linearly with the decreasing step size, this correlation is not as pronounced for the $L^2$ flow;
(ii) the discrete regularity condition appears to be still satisfied for the $H^1$ gradient flow, while it appears to be clearly violated for the $L^2$ flow;
(iii) only the constraint violation for the $H^1$ flow still decays quadratically with the BDF2 method, while for the $L^2$ flow we only observe a rough linear convergence in the constraint, and the energies do not seem to converge at all for both flows;
(iv) all iterations appear to satisfy the stopping criterion at a meta stable state of the (discrete) energy landscape as shown in the middle and right plot in Figure~\ref{fig:random}.
Observation~(i) underlines the preconditioning effects of the $H^1$ scalar product as the natural metric for the minimization problem.
While the $H^1$ flow provides strong control over the time derivative, cf.~Remark~\ref{rem:time_der}, and is less sensitive to rough initial data, (ii) shows that for the irregular initial data the regularity condition is not satisfied in the $L^2$ flow, which explains the constraint error behavior reported in~(iii).
Finally the termination at a meta stable state~(iv), which is possibly caused by larger constraint errors due to the extremely high initial energies, explains the lack of energy convergence in observation~(iv).
We note that the occurrence of distinct meta stable states in the gradient flow appears to be related to specific random initial data and that stopping of the algorithm at such a state was also observed with the Euler method for the same initial data.
With all other parameters unchanged, we did not observe this stopping phenomenon for a different set of random interior nodal values.

\begin{table}[h!]
  {\footnotesize
  \begin{tabular}{ c r c c c c c c c }
    $\tau$ & $N_\stop$ & $\delta_\uni[u_h^\stop]$ & $\mathrm{eoc}_\uni$
      & $A^2$ & $B^2$ & $I_\hm[u_h^\stop]$ & $\delta_\ener[u_h^\stop]$ & $\mathrm{eoc}_\ener$\\ \hline\hline
    \multicolumn{9}{c}{ BDF2 method ($L^2$-gradient flow) }\\
    $2^{-8}$   & 718    & 6.133e+00 &  --- & 1.989e+05 & 1.824e+04 & 666.7 & 6.637e+02 &    ---    \\
    $2^{-9}$   & 458    & 4.203e+00 & 0.55 & 4.908e+05 & 6.762e+04 & 714.7 & 7.117e+02 & -0.10\\
    $2^{-10}$  & 792    & 2.839e+00 & 0.57 & 1.176e+06 & 2.361e+05 & 794.2 & 7.912e+02 & -0.15 \\
    $2^{-11}$  & 4288   & 1.887e+00 & 0.59 & 2.775e+06 & 7.526e+05 & 753.8 & 7.508e+02 & 0.08 \\
    $2^{-12}$  & 8206   & 1.195e+00 & 0.66 & 6.286e+06 & 2.129e+06 & 514.9 & 5.119e+02 & 0.55 \\
    $2^{-13}$  & 5877   & 6.770e-01 & 0.82 & 1.251e+07 & 5.261e+06 & 246.2 & 2.432e+02 & 1.07 \\
    $2^{-14}$  & 27184  & 3.296e-01 & 1.04 & 2.079e+07 & 1.130e+07 & 98.16 & 9.515e+01 & 1.35 \\
    $2^{-15}$  & 22244  & 1.328e-01 & 1.31 & 2.737e+07 & 2.111e+07 & 22.65 & 1.964e+01 & 2.28 \\
    \hline\hline
    \multicolumn{9}{c}{ BDF2 method ($H^1$-gradient flow) }\\
    $2^{-1}$  & 306    & 2.537e-01 &  --- & 4.129e-01 & 1.343e-01 & 28.46 & 2.545e+01 &    ---    \\
    $2^{-2}$  & 426    & 6.806e-02 & 1.90 & 3.928e-01 & 1.934e-01 & 12.00 & 8.994e+00 & 1.50 \\
    $2^{-3}$  & 837    & 1.439e-02 & 2.24 & 2.442e-01 & 2.388e-01 & 10.14 & 7.130e+00 & 0.34 \\
    $2^{-4}$  & 1796   & 3.048e-03 & 2.24 & 1.334e-01 & 2.677e-01 & 9.939 & 6.930e+00 & 0.04 \\
    $2^{-5}$  & 4048   & 6.702e-04 & 2.19 & 6.200e-02 & 2.841e-01 & 9.909 & 6.900e+00 & 0.01 \\
    $2^{-6}$  & 7483   & 1.580e-04 & 2.09 & 3.116e-02 & 2.929e-01 & 9.904 & 6.895e+00 & 0.00 \\
    $2^{-7}$  & 14849  & 3.832e-05 & 2.04 & 1.562e-02 & 2.975e-01 & 9.903 & 6.893e+00 & 0.00 \\
    $2^{-8}$  & 29646  & 9.445e-06 & 2.02 & 7.802e-03 & 2.998e-01 & 9.902 & 6.893e+00 & 0.00 \\
    \hline\hline
  \end{tabular}
   }
  \caption{Step sizes, number of iterations, constraint violation, discrete regularity measures, and energy
    errors for the BDF2 method approximating $L^2$ and $H^1$ gradient flows for
    harmonic maps using random initial data in Example~\ref{ex:invstereo}.}
  \label{tab:randomh1l2}
\end{table}

}

\subsection{Bending isometries}\label{subsec:bending_exp}
Large bending deformations of thin elastic sheets can be determined
via a dimensionally reduced description resulting as a $\G$ limit of
three-dimensional hyperelasticity; cf.~\cite{FrJaMu02}. The variational formulation
seeks a minimizing deformation for the functional
\[
I_\bend(u) = \frac12 \int_\o |D^2 u|^2 \dv{x}
\]
in the set of functions $u \in H^2(\o;\R^3)$ satisfying the pointwise
isometry constraint
\[
(\nabla u)^\transp (\nabla u)  - \id_{2\times 2} = 0,
\]
with the identity matrix $ \id_{2\times 2} \in \R^{2\times 2}$, and the boundary conditions
\[
u|_{\g_D} = u_D, \quad \nabla u|_{\g_D} = \phi_D,
\]
for given functions $u_D\in C(\g_D;\R^3)$ and $\phi_D \in C(\g_D;\R^{3\times 2})$.
Our discretization is based on the nonconforming space of discrete
Kirchhoff triangles and a discrete gradient operator, i.e.,
\[\begin{split}
V_h = \big\{ & v_h \in C(\overline{\o};\R^3): v_h|_T \in P_{3,{\rm red}}(T)^3 \ \text{for all } T\in \cT_h, \\
& \nabla v_h \, \text{continuous in every }z\in \cN_h \big\},
\end{split}\]
where $P_{3,{\rm red}}(T)$ denotes a nine-dimensional subspace of cubic polynomials,
and, with the space of elementwise quadratic, continuous functions $\cS^2(\cT_h)$,
\[
\nabla_h : V_h \to \cS^2(\cT_h)^{3\times 2}.
\]
The matrix of second derivatives $D^2 u$ in $I_{\rm iso}$ is replaced by
the discrete second derivatives $D_h^2 u_h = \nabla \nabla_h u_h$. The isometry constraint is
imposed at the nodes $z\in \cN_h$ of the triangulation; cf.~\cite{Bart13,BarPal22}
for related details. The discretization defines the bilinear form
\[
a_h(u_h,v_h) = \int_\o D_h^2 u_h : D_h^2 v_h \dv{x},
\]
the linear functional
\[
\ell_{\bc,h}(u_h) = \big(u_h|_{\g_D},\nabla_h u_h|_{\g_D}\big),
\]
and the linearized constraint functional evaluated at the nodes of the triangulation
\[
g_h(\hu_h;v_h)
= \big(\big[(\nabla \hu_h)^\transp (\nabla v_h) + (\nabla v_h)^\transp (\nabla \hu_h)\big](z)\big)_{z\in \cN_h}.
\]
Other approaches to the discretization of nonlinear bending problems such as
discontinuous Galerkin methods as devised in~\cite{BGNY23}
can also be formulated in this abstract way.
We test Algorithm~\ref{alg:bdf2_iter_gen} for a setting leading to the formation of a M\"obius strip.

\begin{example}[M\"obius strip]\label{ex:moebius}
Let $\o = (0,L)\times (-w/2,w/2)$ and $\g_D = \{0,L\}\times [-w/2,w/2]$ with $L = 12$ and $w=2$.
We choose boundary data $u_D$ and $\phi_D$ that map the two sides contained in $\gamma_D$ to the
same interval but enforce a half-rotation of the strip~$\o$.
\end{example}

 \begin{figure}[p]
    \centering
    \includegraphics[width=0.24\linewidth]{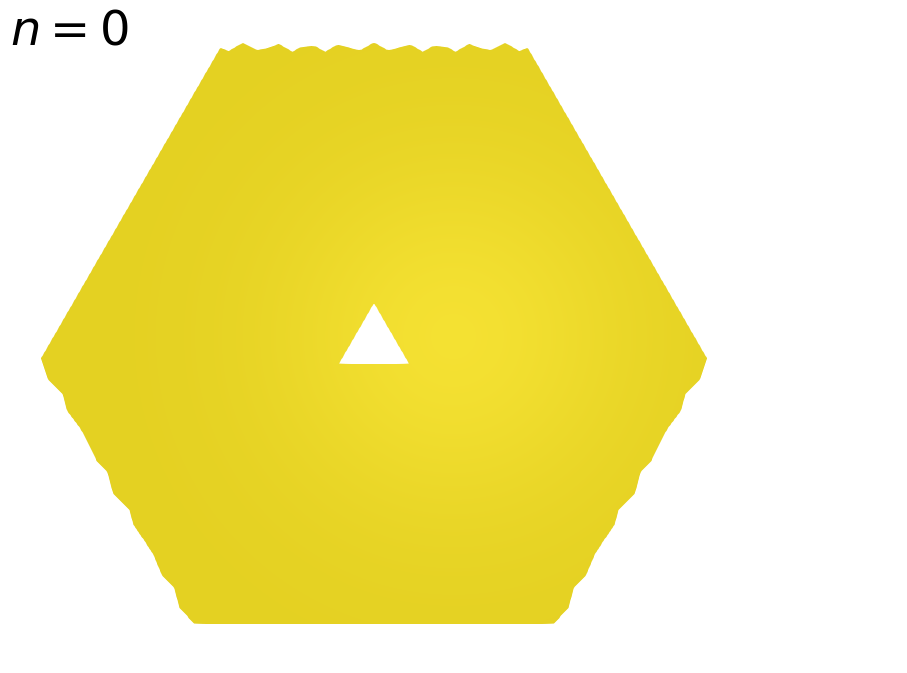}
    \includegraphics[width=0.24\linewidth]{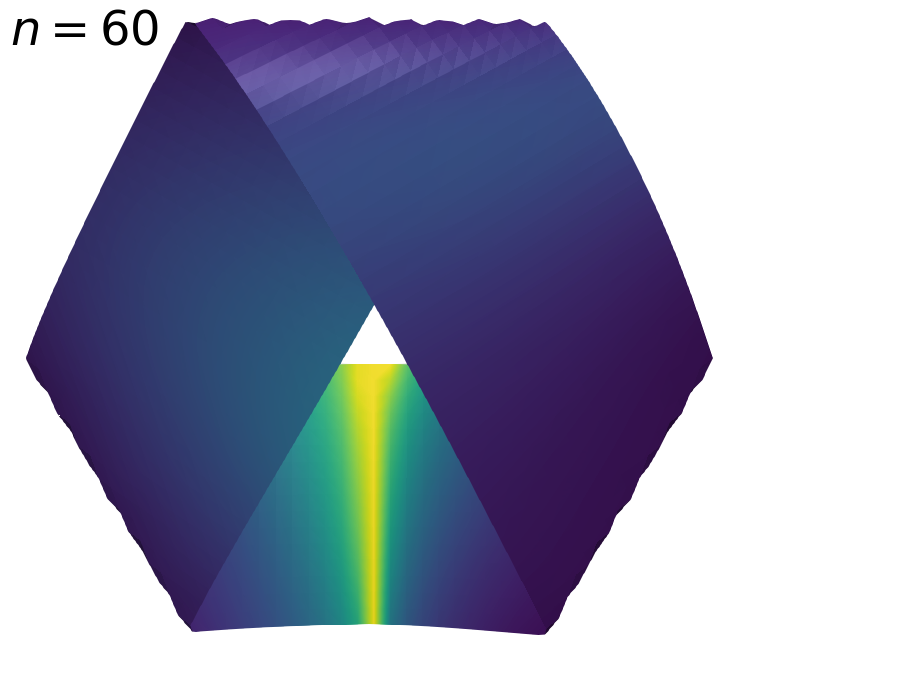}
    \includegraphics[width=0.24\linewidth]{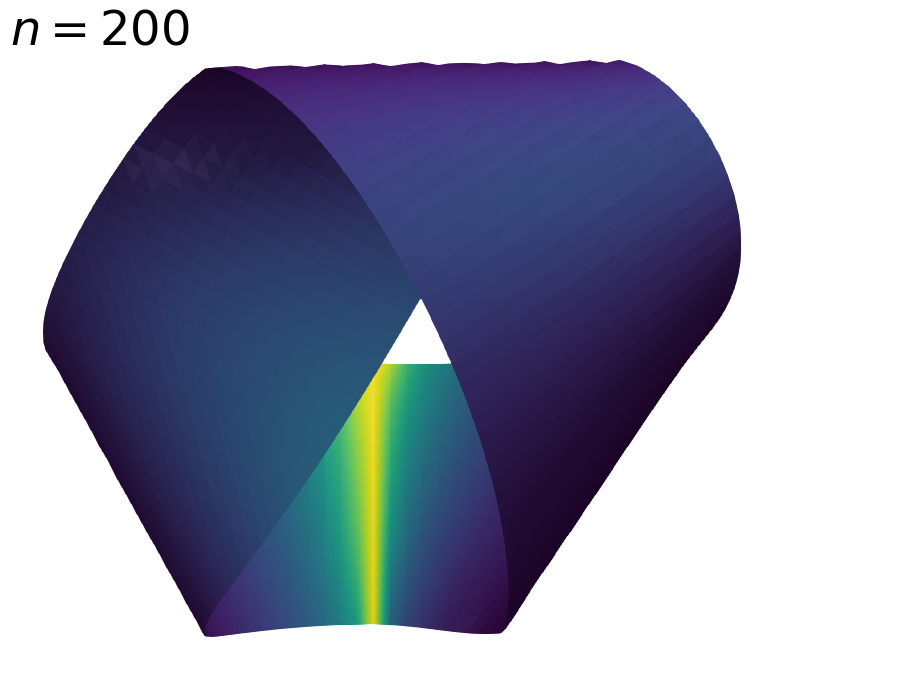}
    \includegraphics[width=0.24\linewidth]{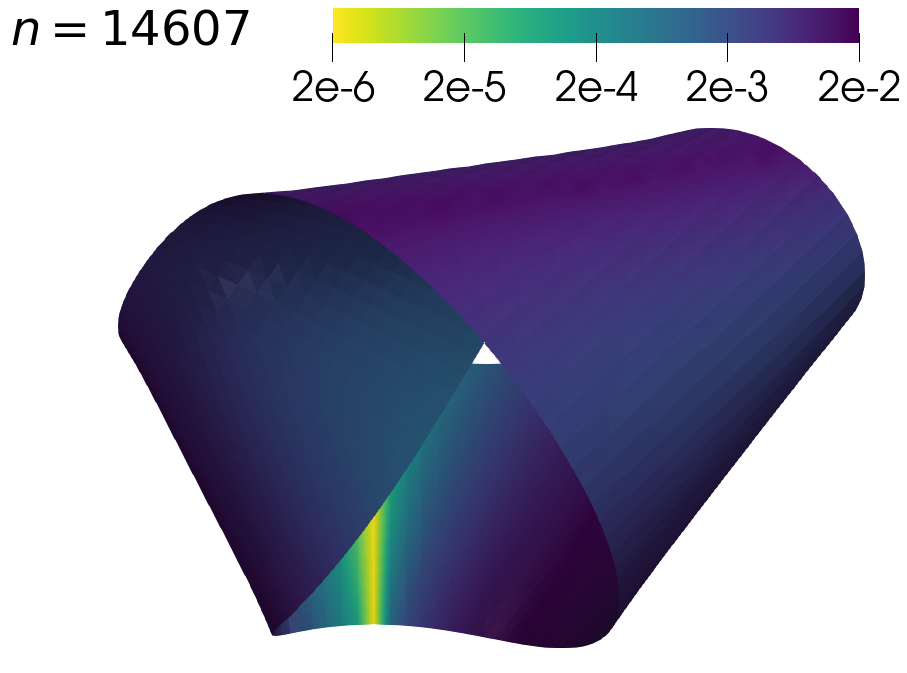}\\
    \includegraphics[width=0.24\linewidth]{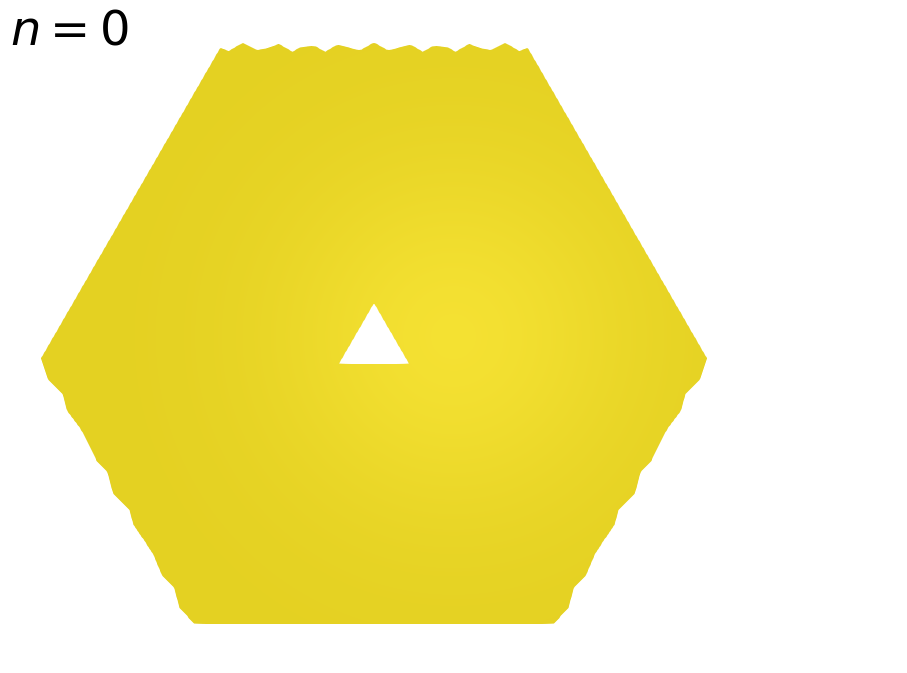}
    \includegraphics[width=0.24\linewidth]{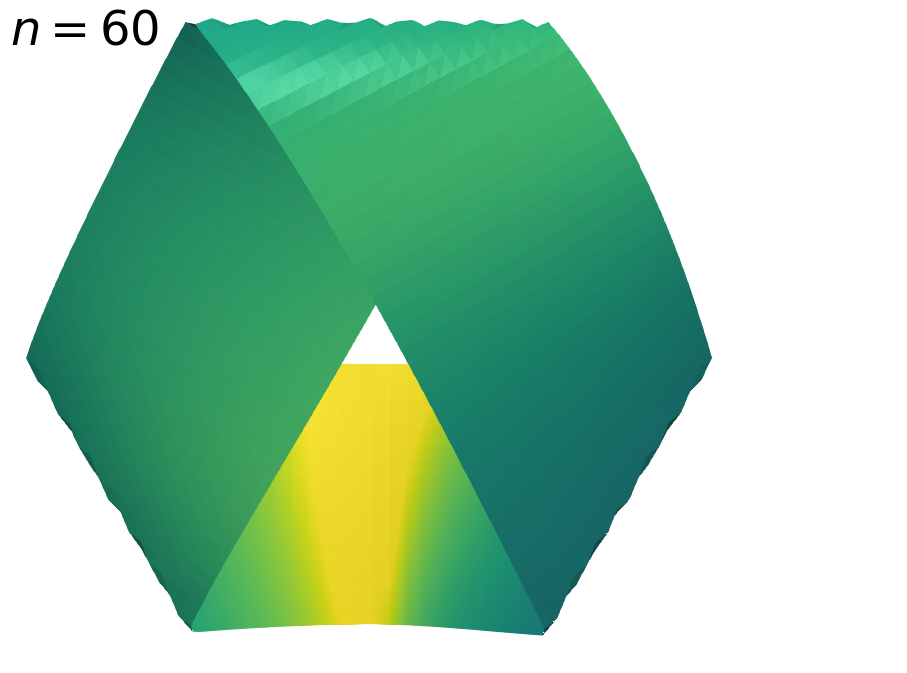}
    \includegraphics[width=0.24\linewidth]{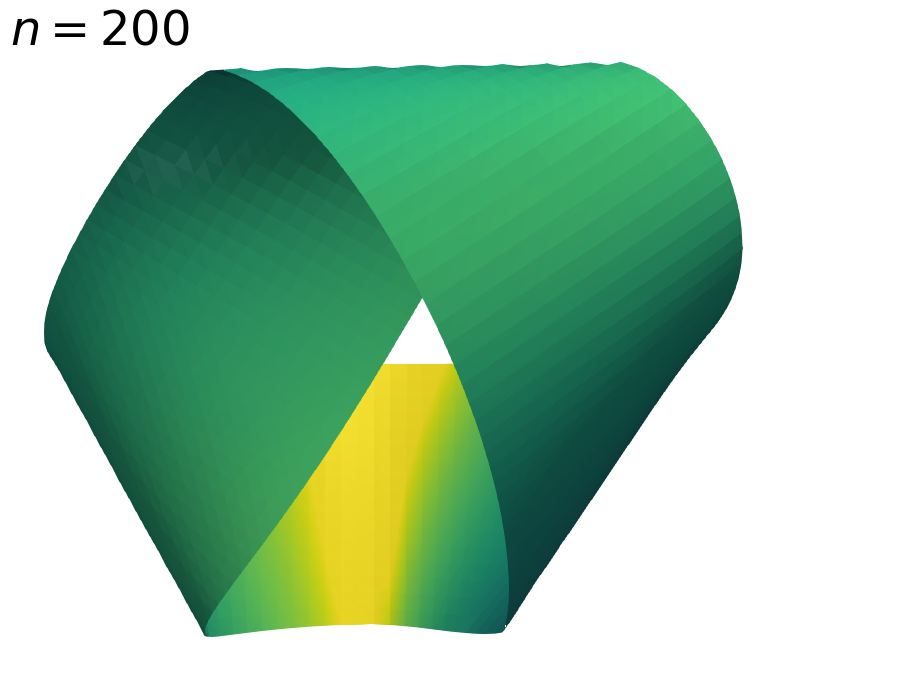}
    \includegraphics[width=0.24\linewidth]{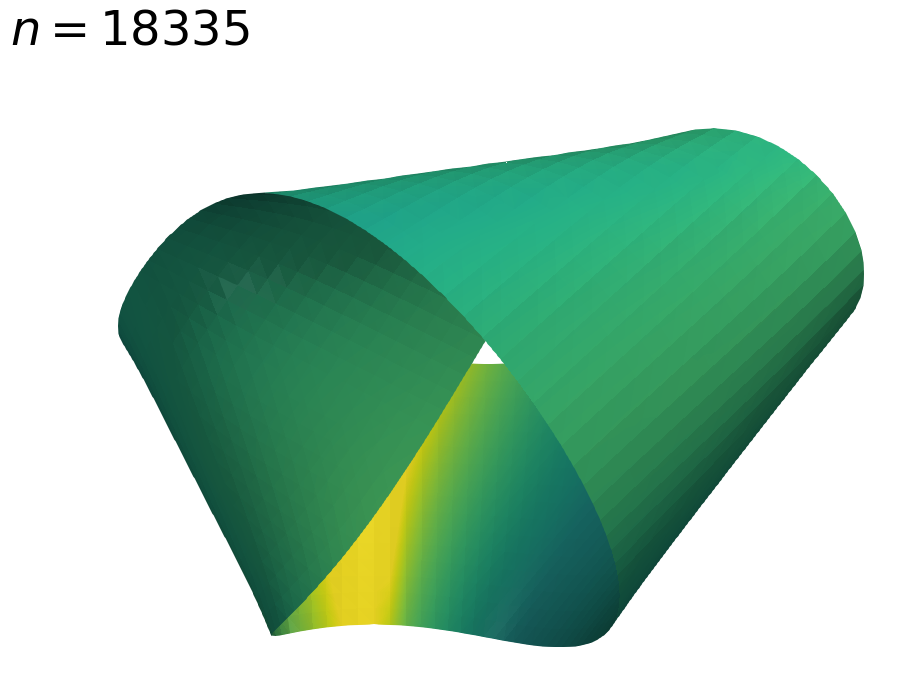}
    \caption{Evolution of an initially flat M\"obius strip from Example~\ref{ex:moebius} using
      the implicit Euler (top row) and BDF2 methods (bottom row) realizing discrete $H^2$-gradient
      flows with step sizes $\tau=2^{-7}$. The coloring represents the constraint violation.}
    \label{fig:moebius}
  \end{figure}

 \begin{table}[p]
    {\footnotesize
      \begin{tabular}{ c r c c c c c }
      $\tau$ & $N_\stop$ & $\delta_\iso[u_h^\stop]$ & $\mathrm{eoc}_\iso$
        & $A^2$ & $B^2$ & $I_\bend[u_h^\stop]$ \\ \hline\hline
      \multicolumn{7}{c}{ Implicit Euler method ($H^2_h$-gradient flow)}\\
      $2^{-0}$  & 349    & 2.255e+01   &  ---  & 4.617e+02 & 1.054e--01 & 14.98 \\
      $2^{-1}$  & 198    & 1.342e+01   & 0.75  & 1.038e+03 & 1.874e--01 & 12.37 \\
      $2^{-2}$  & 446    & 5.243e+00   & 1.36  & 9.363e+02 & 2.699e--01 & 10.95 \\
      $2^{-3}$  & 874    & 2.493e+00   & 1.07  & 9.736e+02 & 3.332e--01 & 10.29 \\
      $2^{-4}$  & 1791   & 1.226e+00   & 1.02  & 8.741e+02 & 3.735e--01 & 9.983 \\
      $2^{-5}$  & 3617   & 6.142e--01  & 1.00  & 6.839e+02 & 3.965e--01 & 9.840 \\
      $2^{-6}$  & 7261   & 3.072e--01  & 1.00  & 4.351e+02 & 4.088e--01 & 9.768 \\
      $2^{-7}$  & 14538  & 1.538e--01  & 1.00  & 2.492e+02 & 4.151e--01 & 9.732 \\
      $2^{-9}$  & 29086  & 7.695e--02  & 1.00  & 1.339e+02 & 4.184e--01 & 9.715 \\ \hline
%
      \multicolumn{7}{c}{ BDF2 method ($H^2_h$-gradient flow)}\\
      $2^{-0}$  & 151    & 5.993e+00   &  --- & 5.970e+01 & 1.054e--01 & 11.63 \\
      $2^{-1}$  & 263    & 6.713e+00   &-0.16 & 2.789e+02 & 1.874e--01 & 11.04 \\
      $2^{-2}$  & 490    & 2.552e+00   & 1.40 & 4.090e+02 & 2.699e--01 & 10.05 \\
      $2^{-3}$  & 1049   & 9.924e--01  & 1.36 & 6.480e+02 & 3.332e--01 & 9.834 \\
      $2^{-4}$  & 2273   & 1.905e--01  & 2.38 & 4.980e+02 & 3.735e--01 & 9.727 \\
      $2^{-5}$  & 4578   & 4.827e--02  & 1.98 & 5.081e+02 & 3.965e--01 & 9.704 \\
      $2^{-6}$  & 9166   & 9.470e--03  & 2.35 & 3.986e+02 & 4.088e--01 & 9.698 \\
      $2^{-7}$  & 18313  & 1.506e--03  & 2.65 & 2.535e+02 & 4.151e--01 & 9.697 \\
      $2^{-9}$  & 36590  & 2.065e--04  & 2.87 & 1.390e+02 & 4.184e--01 & 9.697 \\  \hline\hline
    \end{tabular}
      }
    \caption{Step sizes, number of iterations, constraint violation, discrete regularity measures,
      and energies for the implicit Euler and BDF2 methods approximating a discrete $H^2$ gradient flow
      leading to the formation of a M\"obius strip for the boundary conditions specified in Example~\ref{ex:moebius}.}
    \label{tab:moebiush2}
  \end{table}


As initial data $u^0$ that is compatible with the boundary conditions and isometry constraint
we use a Lipschitz continuous function that defines a flat, folded M\"obius strip. The
interpolated function $u_h^0$ on a triangulation of $\o$ into 3072 triangles resembling halved squares
is shown in Figure~\ref{fig:moebius}. The initial data is
thus of unbounded bending energy as the mesh-size tends to zero. Using the bilinear form
$a_h$ to define discrete $H^2$ gradient flows determined by
the implicit Euler and BDF2 methods we obtained the iterates shown also in Figure~\ref{fig:moebius}.
The unfolding of the initially flat configuration was obtained with a forcing term in the energy
that was set to zero for $t_n \ge t_f = 2$. From the coloring used for the plots in the figure we
observe that the BDF2 method leads to significantly
reduced constraint errors. This observation is confirmed by the numbers displayed in
Table~\ref{tab:moebiush2}. For the implicit Euler and the BDF2 methods we computed the
isometry constraint violation errors
\[
\delta_\iso[u_h] = \big\| \cI_h\big(|(\nabla u_h)^\transp (\nabla u_h) - \id_{2 \times 2}|\big) \big\|_{L^1},
\]
and the discrete regularity quantities
\[
A^2 =  \tau^2 \sum_{n=2}^{N'} \|\nabla d_t^2 u_h^n\|^2, \quad B^2 = \|\nabla d_t u_h^1 \|^2,
\]
with $N'=N_\stop$. Correspondingly, we used $\|v_h\|_\sharp = \|\nabla v_h\|$ to evaluate the stopping criterion
with $\veps_\stop = 10^{-3}$. Our overall observations are similar to those for the approximation of
harmonic maps using an $H^1$ gradient flow. In particular, we find that (i)~the number of
iterations needed to satisfy the stopping criterion grow linearly with $\tau^{-1}$ and
are comparable for the implicit Euler and BDF2 methods, (ii)~the constraint violation decays
significantly faster for the BDF2 method than for the implicit Euler method and the discrete
energies are lower, and (iii)~the discrete regularity quantities remain bounded as the
step sizes are reduced.

\subsection*{Acknowledgments}
Financial support by the German Research Foundation (DFG)
via research unit FOR 3013 {\em Vector- and tensor-valued surface PDEs}
(Grant no. BA2268/6–1) is gratefully acknowledged.


\section*{References}
\printbibliography[heading=none]

@book {Thom06-book,
    AUTHOR = {Thom\'{e}e, Vidar},
     TITLE = {Galerkin finite element methods for parabolic problems},
    SERIES = {Springer Series in Computational Mathematics},
    VOLUME = {25},
   EDITION = {Second},
 PUBLISHER = {Springer-Verlag, Berlin},
      YEAR = {2006},
     PAGES = {xii+370},
      ISBN = {978-3-540-33121-6; 3-540-33121-2},
   MRCLASS = {65-02 (65M15 65M60)},
  MRNUMBER = {2249024},
}

@article {GutRes17,
    AUTHOR = {Guti\'{e}rrez-Santacreu, Juan Vicente and Restelli, Marco},
     TITLE = {Inf-sup stable finite element methods for the
              {L}andau-{L}ifshitz-{G}ilbert and harmonic map heat flow
              equations},
   JOURNAL = {SIAM J. Numer. Anal.},
  FJOURNAL = {SIAM Journal on Numerical Analysis},
    VOLUME = {55},
      YEAR = {2017},
    NUMBER = {6},
     PAGES = {2565--2591},
      ISSN = {0036-1429},
   MRCLASS = {65M60 (35K55 35Q60 65M12)},
  MRNUMBER = {3719027},
MRREVIEWER = {Severiano Gonzalez-Pinto},
       DOI = {10.1137/17M1116799},
       URL = {https://doi.org/10.1137/17M1116799},
}

@article {BadGuiGut11,
    AUTHOR = {Badia, Santiago and Guill\'{e}n-Gonz\'{a}lez, Francisco and
              Guti\'{e}rrez-Santacreu, Juan Vicente},
     TITLE = {Finite element approximation of nematic liquid crystal flows
              using a saddle-point structure},
   JOURNAL = {J. Comput. Phys.},
  FJOURNAL = {Journal of Computational Physics},
    VOLUME = {230},
      YEAR = {2011},
    NUMBER = {4},
     PAGES = {1686--1706},
      ISSN = {0021-9991},
   MRCLASS = {76A15 (76M10)},
  MRNUMBER = {2753385},
       DOI = {10.1016/j.jcp.2010.11.033},
       URL = {https://doi.org/10.1016/j.jcp.2010.11.033},
}

@article {QiTaWi09,
    AUTHOR = {Hu, Qiya and Tai, Xue-Cheng and Winther, Ragnar},
     TITLE = {A saddle point approach to the computation of harmonic maps},
   JOURNAL = {SIAM J. Numer. Anal.},
  FJOURNAL = {SIAM Journal on Numerical Analysis},
    VOLUME = {47},
      YEAR = {2009},
    NUMBER = {2},
     PAGES = {1500--1523},
      ISSN = {0036-1429},
   MRCLASS = {35J62 (35J50 35J57 65N30)},
  MRNUMBER = {2497338},
       DOI = {10.1137/060675575},
       URL = {https://doi.org/10.1137/060675575},
}

@article {Alou97,
    AUTHOR = {Alouges, Fran\c{c}ois},
     TITLE = {A new algorithm for computing liquid crystal stable
              configurations: the harmonic mapping case},
   JOURNAL = {SIAM J. Numer. Anal.},
  FJOURNAL = {SIAM Journal on Numerical Analysis},
    VOLUME = {34},
      YEAR = {1997},
    NUMBER = {5},
     PAGES = {1708--1726},
      ISSN = {0036-1429},
   MRCLASS = {82D30 (65C20 76A15 76M25)},
  MRNUMBER = {1472192},
MRREVIEWER = {Denis Serre},
       DOI = {10.1137/S0036142994264249},
       URL = {https://doi.org/10.1137/S0036142994264249},
}

@article {Bart05,
    AUTHOR = {Bartels, S\"{o}ren},
     TITLE = {Stability and convergence of finite-element approximation
              schemes for harmonic maps},
   JOURNAL = {SIAM J. Numer. Anal.},
  FJOURNAL = {SIAM Journal on Numerical Analysis},
    VOLUME = {43},
      YEAR = {2005},
    NUMBER = {1},
     PAGES = {220--238},
      ISSN = {0036-1429},
   MRCLASS = {65N30 (58E20 76A15 82D30)},
  MRNUMBER = {2177142},
MRREVIEWER = {R. Kodn\'{a}r},
       DOI = {10.1137/040606594},
       URL = {https://doi.org/10.1137/040606594},
}

@book {Bart15-book,
    AUTHOR = {Bartels, S\"{o}ren},
     TITLE = {Numerical methods for nonlinear partial differential
              equations},
    SERIES = {Springer Series in Computational Mathematics},
    VOLUME = {47},
 PUBLISHER = {Springer, Cham},
      YEAR = {2015},
     PAGES = {x+393},
   MRCLASS = {65-01 (35A15 35A35 65Mxx 65Nxx)},
  MRNUMBER = {3309171},
MRREVIEWER = {Karsten Urban},
       DOI = {10.1007/978-3-319-13797-1},
       URL = {https://doi.org/10.1007/978-3-319-13797-1},
}

@article {Rivi95,
    AUTHOR = {Rivi\`ere, Tristan},
     TITLE = {Everywhere discontinuous harmonic maps into spheres},
   JOURNAL = {Acta Math.},
  FJOURNAL = {Acta Mathematica},
    VOLUME = {175},
      YEAR = {1995},
    NUMBER = {2},
     PAGES = {197--226},
      ISSN = {0001-5962},
   MRCLASS = {58E20 (49N60)},
  MRNUMBER = {1368247},
MRREVIEWER = {Martin Fuchs},
       DOI = {10.1007/BF02393305},
       URL = {https://doi.org/10.1007/BF02393305},
}

@article {KPPRS19,
    AUTHOR = {Kraus, Johannes and Pfeiler, Carl-Martin and Praetorius, Dirk
              and Ruggeri, Michele and Stiftner, Bernhard},
     TITLE = {Iterative solution and preconditioning for the tangent plane
              scheme in computational micromagnetics},
   JOURNAL = {J. Comput. Phys.},
  FJOURNAL = {Journal of Computational Physics},
    VOLUME = {398},
      YEAR = {2019},
     PAGES = {108866, 27},
      ISSN = {0021-9991},
   MRCLASS = {65M99 (65Z05)},
  MRNUMBER = {3995784},
MRREVIEWER = {Khaled Mohammed Saad},
       DOI = {10.1016/j.jcp.2019.108866},
       URL = {https://doi.org/10.1016/j.jcp.2019.108866},
}

@article {Bart16,
    AUTHOR = {Bartels, S\"{o}ren},
     TITLE = {Projection-free approximation of geometrically constrained
              partial differential equations},
   JOURNAL = {Math. Comp.},
  FJOURNAL = {Mathematics of Computation},
    VOLUME = {85},
      YEAR = {2016},
    NUMBER = {299},
     PAGES = {1033--1049},
      ISSN = {0025-5718},
   MRCLASS = {65J05 (65M12 65M60 65N12 65N30)},
  MRNUMBER = {3454357},
MRREVIEWER = {Waleed M. Abd-Elhameed},
       DOI = {10.1090/mcom/3008},
       URL = {https://doi.org/10.1090/mcom/3008},
}

@article {AFKL21,
    AUTHOR = {Akrivis, Georgios and Feischl, Michael and Kov\'{a}cs, Bal\'{a}zs and
              Lubich, Christian},
     TITLE = {Higher-order linearly implicit full discretization of the
              {L}andau-{L}ifshitz-{G}ilbert equation},
   JOURNAL = {Math. Comp.},
  FJOURNAL = {Mathematics of Computation},
    VOLUME = {90},
      YEAR = {2021},
    NUMBER = {329},
     PAGES = {995--1038},
      ISSN = {0025-5718},
   MRCLASS = {65M60 (35Q60 65L06 65M12 65M15)},
  MRNUMBER = {4232216},
MRREVIEWER = {Hamdullah Y\"{u}cel},
       DOI = {10.1090/mcom/3597},
       URL = {https://doi.org/10.1090/mcom/3597},
}

@article {BarPal22,
    AUTHOR = {Bartels, S\"{o}ren and Palus, Christian},
     TITLE = {Stable gradient flow discretizations for simulating bilayer
              plate bending with isometry and obstacle constraints},
   JOURNAL = {IMA J. Numer. Anal.},
  FJOURNAL = {IMA Journal of Numerical Analysis},
    VOLUME = {42},
      YEAR = {2022},
    NUMBER = {3},
     PAGES = {1903--1928},
      ISSN = {0272-4979},
   MRCLASS = {65N30 (65N12)},
  MRNUMBER = {4454912},
       DOI = {10.1093/imanum/drab050},
       URL = {https://doi.org/10.1093/imanum/drab050},
}

@article {Bart13,
    AUTHOR = {Bartels, S\"{o}ren},
     TITLE = {Approximation of large bending isometries with discrete
              {K}irchhoff triangles},
   JOURNAL = {SIAM J. Numer. Anal.},
  FJOURNAL = {SIAM Journal on Numerical Analysis},
    VOLUME = {51},
      YEAR = {2013},
    NUMBER = {1},
     PAGES = {516--525},
      ISSN = {0036-1429},
   MRCLASS = {65N30 (74S05)},
  MRNUMBER = {3033021},
MRREVIEWER = {Josef Dan\v{e}k},
       DOI = {10.1137/110855405},
       URL = {https://doi.org/10.1137/110855405},
}

@book {HaiWan96,
    AUTHOR = {Hairer, E. and Wanner, G.},
     TITLE = {Solving ordinary differential equations II: stiff and differential-algebraic problems},
    SERIES = {Springer Series in Computational Mathematics},
    VOLUME = {14},
   EDITION = {Second},
 PUBLISHER = {Springer-Verlag, Berlin},
      YEAR = {1996},
     PAGES = {xvi+614},
      ISBN = {3-540-60452-9},
   MRCLASS = {65-02 (34A09 34A45 65-01 65Lxx)},
  MRNUMBER = {1439506},
       DOI = {10.1007/978-3-642-05221-7},
       URL = {https://doi.org/10.1007/978-3-642-05221-7},
}

@article{BaKoWa23,
    author = {Bartels, Sören and Kovács, Balázs and Wang, Zhangxian},
    title = "{Error analysis for the numerical approximation of the harmonic map heat flow with nodal constraints}",
    journal = {IMA J.~Numer.~Anal.},
    volume = {44},
    number = {2},
    pages = {633-653},
    year = {2023},
    month = {06},
    issn = {0272-4979},
    doi = {10.1093/imanum/drad037},
    url = {https://doi.org/10.1093/imanum/drad037},
}

@article {AKST14,
    AUTHOR = {Alouges, Fran\c{c}ois and Kritsikis, Evaggelos and Steiner, Jutta
              and Toussaint, Jean-Christophe},
     TITLE = {A convergent and precise finite element scheme for
              {L}andau-{L}ifschitz-{G}ilbert equation},
   JOURNAL = {Numer. Math.},
  FJOURNAL = {Numerische Mathematik},
    VOLUME = {128},
      YEAR = {2014},
    NUMBER = {3},
     PAGES = {407--430},
      ISSN = {0029-599X},
   MRCLASS = {65M60 (35K55 35Q60 65M12)},
  MRNUMBER = {3268842},
MRREVIEWER = {Nicolae Pop},
       DOI = {10.1007/s00211-014-0615-3},
       URL = {https://doi.org/10.1007/s00211-014-0615-3},
}

@article {AnGaSu21,
    AUTHOR = {An, Rong and Gao, Huadong and Sun, Weiwei},
     TITLE = {Optimal error analysis of {E}uler and {C}rank-{N}icolson
              projection finite difference schemes for {L}andau-{L}ifshitz
              equation},
   JOURNAL = {SIAM J. Numer. Anal.},
  FJOURNAL = {SIAM Journal on Numerical Analysis},
    VOLUME = {59},
      YEAR = {2021},
    NUMBER = {3},
     PAGES = {1639--1662},
      ISSN = {0036-1429},
   MRCLASS = {65M06 (35K61 65M12 65M15)},
  MRNUMBER = {4272916},
       DOI = {10.1137/20M1335431},
       URL = {https://doi.org/10.1137/20M1335431},
}

@article {FPPRS20,
    AUTHOR = {Di Fratta, Giovanni and Pfeiler, Carl-Martin and Praetorius,
              Dirk and Ruggeri, Michele and Stiftner, Bernhard},
     TITLE = {Linear second-order {IMEX}-type integrator for the (eddy
              current) {L}andau-{L}ifshitz-{G}ilbert equation},
   JOURNAL = {IMA J. Numer. Anal.},
  FJOURNAL = {IMA Journal of Numerical Analysis},
    VOLUME = {40},
      YEAR = {2020},
    NUMBER = {4},
     PAGES = {2802--2838},
      ISSN = {0272-4979},
   MRCLASS = {65M60 (65M12 78M10)},
  MRNUMBER = {4167063},
       DOI = {10.1093/imanum/drz046},
       URL = {https://doi.org/10.1093/imanum/drz046},
}

@article {GuLiWa22,
    AUTHOR = {Gui, Xinping and Li, Buyang and Wang, Jilu},
     TITLE = {Convergence of renormalized finite element methods for heat
              flow of harmonic maps},
   JOURNAL = {SIAM J. Numer. Anal.},
  FJOURNAL = {SIAM Journal on Numerical Analysis},
    VOLUME = {60},
      YEAR = {2022},
    NUMBER = {1},
     PAGES = {312--338},
      ISSN = {0036-1429},
   MRCLASS = {65M60 (35K55 35Q35 65M12)},
  MRNUMBER = {4377027},
       DOI = {10.1137/21M1402212},
       URL = {https://doi.org/10.1137/21M1402212},
}

@article {MPPR22,
    AUTHOR = {Mauser, Norbert J. and Pfeiler, Carl-Martin and Praetorius,
              Dirk and Ruggeri, Michele},
     TITLE = {Unconditional well-posedness and {IMEX} improvement of a
              family of predictor-corrector methods in micromagnetics},
   JOURNAL = {Appl. Numer. Math.},
  FJOURNAL = {Applied Numerical Mathematics. An IMACS Journal},
    VOLUME = {180},
      YEAR = {2022},
     PAGES = {33--54},
      ISSN = {0168-9274},
   MRCLASS = {65M60 (65M12 78M10)},
  MRNUMBER = {4432060},
       DOI = {10.1016/j.apnum.2022.05.008},
       URL = {https://doi.org/10.1016/j.apnum.2022.05.008},
}

@article {BarPro07,
    AUTHOR = {Bartels, S\"{o}ren and Prohl, Andreas},
     TITLE = {Constraint preserving implicit finite element discretization
              of harmonic map flow into spheres},
   JOURNAL = {Math. Comp.},
  FJOURNAL = {Mathematics of Computation},
    VOLUME = {76},
      YEAR = {2007},
    NUMBER = {260},
     PAGES = {1847--1859},
      ISSN = {0025-5718},
   MRCLASS = {65M60 (35K55 53C44 58E20)},
  MRNUMBER = {2336271},
MRREVIEWER = {Beny Neta},
       DOI = {10.1090/S0025-5718-07-02026-1},
       URL = {https://doi.org/10.1090/S0025-5718-07-02026-1},
}

@article {BGNY23,
    AUTHOR = {Bonito, Andrea and Guignard, Diane and Nochetto, Ricardo H.
              and Yang, Shuo},
     TITLE = {Numerical analysis of the {LDG} method for large deformations
              of prestrained plates},
   JOURNAL = {IMA J. Numer. Anal.},
  FJOURNAL = {IMA Journal of Numerical Analysis},
    VOLUME = {43},
      YEAR = {2023},
    NUMBER = {2},
     PAGES = {627--662},
      ISSN = {0272-4979},
   MRCLASS = {65M60 (74K20)},
  MRNUMBER = {4568427},
       DOI = {10.1093/imanum/drab103},
       URL = {https://doi.org/10.1093/imanum/drab103},
}

@article {AkrLub15,
    AUTHOR = {Akrivis, Georgios and Lubich, Christian},
     TITLE = {Fully implicit, linearly implicit and implicit-explicit
              backward difference formulae for quasi-linear parabolic
              equations},
   JOURNAL = {Numer. Math.},
  FJOURNAL = {Numerische Mathematik},
    VOLUME = {131},
      YEAR = {2015},
    NUMBER = {4},
     PAGES = {713--735},
      ISSN = {0029-599X},
   MRCLASS = {65M60 (65L06 65M12 65M20)},
  MRNUMBER = {3422451},
MRREVIEWER = {B\"{u}lent Karas\"{o}zen},
       DOI = {10.1007/s00211-015-0702-0},
       URL = {https://doi.org/10.1007/s00211-015-0702-0},
}

@article {FrJaMu02,
    AUTHOR = {Friesecke, Gero and James, Richard D. and M\"{u}ller, Stefan},
     TITLE = {A theorem on geometric rigidity and the derivation of
              nonlinear plate theory from three-dimensional elasticity},
   JOURNAL = {Comm. Pure Appl. Math.},
  FJOURNAL = {Communications on Pure and Applied Mathematics},
    VOLUME = {55},
      YEAR = {2002},
    NUMBER = {11},
     PAGES = {1461--1506},
      ISSN = {0010-3640},
   MRCLASS = {74K20 (49J45 74B20 74G65)},
  MRNUMBER = {1916989},
MRREVIEWER = {Georg K. Dolzmann},
       DOI = {10.1002/cpa.10048},
       URL = {https://doi.org/10.1002/cpa.10048},
}

\end{document}